\documentclass[12pt,reqno]{amsart}
\usepackage[normalem]{ulem}
\usepackage{amsmath,amssymb,amsfonts,amscd,latexsym,amsthm,mathrsfs,verbatim,fullpage, soul}
\usepackage{tabularx}
\usepackage[figuresright]{rotating}

\author{Ling Long, Fang-Ting Tu, Noriko Yui, Wadim Zudilin}

\title[Supercongruences for Hypergeometric Calabi--Yau Threefolds]%
{Supercongruences for Rigid Hypergeometric Calabi--Yau~Threefolds}

\address{Louisiana State University, Baton Rouge, LA 70803, USA}
\email{llong@lsu.edu}

\address{Louisiana State University, Baton Rouge, LA 70803, USA}
\email{ftu@lsu.edu}

\address{Department of Mathematics and Statistics, Queen's University, Kingston, ON, K7L3N6, Canada}
\email{yui@queensu.ca}

\address{Department of Mathematics, IMAPP, Radboud University, PO Box 9010, 6500 GL Nijmegen, The Netherlands}
\email{w.zudilin@math.ru.nl}

\address{Laboratory of Mirror Symmetry and Automorphic Forms, National Research University Higher School of Economics, 6 Usacheva str., 119048 Moscow, Russia}
\email{wzudilin@gmail.com}

\thanks{Long is supported by NSF DMS \#1602047 and Tu is supported by a start-up fund at Louisiana State University.
	Both Long and Tu received travel funding from NSF DMS \#1642598 to attend the MATRIX workshop, Australia.
	Yui is partially supported by a Discovery Grant from NSERC, Canada.
	Zudilin was supported by funding from Queen's University to attend BIRS \#16w5009 workshop which initiated this project.
	Zudilin is partially supported by Laboratory of Mirror Symmetry NRU HSE, RF government grant, ag.\ no.\ 14.641.31.0001.}


\usepackage{amsmath, amssymb, amsthm,color, graphicx, enumitem, hyperref,epsfig}

\newcommand{\Res}{\operatorname{Res}}
\let\wt\widetilde
\newcommand\balpha{{\boldsymbol\alpha}}
\newcommand\bbeta{{\boldsymbol\beta}}

\newtheorem{Theorem}{Theorem}
\newtheorem{Lemma}{Lemma}

\newtheorem{Proposition}{Proposition}

\theoremstyle{definition}

\catcode`,\active

\catcode`\,12

\def\PFQ#1#2#3#4#5{{}_{#1}F_{#2}\left[\begin{matrix}#3 \\[1.5pt] #4\end{matrix} \, ; \, #5\right]}

\usepackage[yyyymmdd]{datetime}
\def\[{\left[}
\def\]{\right]}
\def\({\left(}
\def\){\right)}

\def \ff{\mathfrak f}

\def\C{\mathbb{C}}

\def\Z{\mathbb{Z}}
\def\Q{\mathbb{Q}}
\def\Ql{\mathbb{Q}_{\ell}}
\def\F{\mathbb{F}}
\def\O{\mathcal{O}}

\def\P{\mathbb{P}}

\def \l {\lambda}

\def\lf{\lfloor}
\def\rf{\rfloor}
\def \eps{\varepsilon}
\def\G{\Gamma}

\renewcommand{\d}{{\mathrm d}}
\newcommand{\Frob}{\operatorname{Frob}}

\newcommand \cy {\color{black}}

\newcommand \bk {\color{black}}
\date{\today}

\begin{document}
	
	\begin{abstract}
		We establish the supercongruences for the fourteen rigid hypergeometric Calabi--Yau threefolds over $\mathbb Q$ conjectured by Rodriguez-Villegas in 2003.
		Our first method is based on Dwork's theory of $p$-adic unit roots and it allows us to establish the supercongruences \cy between the truncated hypergeometric series and the corresponding unit roots \bk for ordinary primes.
		The other method makes use of the theory of hypergeometric motives, in particular, adapts
		the techniques from the recent work of Beukers, Cohen and Mellit on finite hypergeometric sums over $\mathbb Q$.
		Essential ingredients in executing the both approaches are the modularity of the underlying Calabi--Yau threefolds
		and a $p$-adic perturbation method applied to hypergeometric functions.
	\end{abstract}
	
	\maketitle

	\section{Introduction}
	\label{sec1}
	
	The purpose of this paper is to establish the supercongruences satisfied by fourteen rigid Calabi--Yau threefolds. Each of those threefolds is a particular instance
	of a one-parameter ``hypergeometric'' family of Calabi--Yau threefolds\,---\,a family whose periods are solutions of the hypergeometric
	equation with parameters $\balpha=\{r_1, 1-r_1, r_2, 1-r_2\}$ and $\bbeta=\{1,1,1,1\}$ viewed as multi-sets,
	where the fourteen possible rational pairs $(r_1,r_2)$ are listed in Table~\ref{tab1} below.
	The corresponding families $V_{\balpha}(\psi)$
	are realized as either one-parameter families of hypersurfaces in weighted projective spaces
	or complete intersections of several such families of hypersurfaces (all explicitly recorded in Tables~\ref{tab3} and~\ref{tab4} below);
	see \cite{Candelas, KT, LT, Morrison} for related details.
	The mirrors of these Calabi--Yau threefolds exist and they are one-parameter families $\hat{\mathcal V}_\balpha(\lambda)$ of Calabi--Yau threefolds with
	the Hodge number $h^{2,1}(\hat{\mathcal V}_\balpha(\lambda))=1$ for generic~$\lambda$. 
	In all the fourteen cases, the Picard--Fuchs differential equations of $\hat{\mathcal V}_\balpha(\lambda)$
	are precisely the \emph{same} order 4 hypergeometric differential equations attached to the data $\balpha$ and $\bbeta$
	(see \cite{ACDZ} and \cite{CYY} for more information on this aspect).
	Their particular analytical solution at the maximally unipotent point\,---\,singularity $\l=0$\,---\,is given by the hypergeometric function
	$$
	\PFQ{4}{3}{\balpha}{\bbeta}{\l}
	=\PFQ{4}{3}{\alpha_1 ,\, \alpha_2 ,\, \alpha_3 ,\, \alpha_4}{\; 1 ,\; 1 ,\; 1}{\l}
	:=\sum_{k=0}^\infty\frac{(\alpha_1)_k(\alpha_2)_k(\alpha_3)_k(\alpha_4)_k}{k!^4}\,\l^k,
	$$
	where $(\alpha)_k:=\G(\alpha+k)/\G(\alpha)=\alpha(\alpha+1)\dotsb(\alpha+k-1)$, $(\alpha)_0:=1$ denotes the Pochhammer symbol (rising factorial) and $\G(x)$ is the standard Gamma function;
	a general definition of classical hypergeometric functions is given in \eqref{eq:HGS} below.
	
	One of the most known and developed examples is the quintic threefold
	$$
	V(\psi):=V_{\{\frac 15,\frac 25,\frac 35,\frac 45\}}(\psi):\quad X_1^5+X_2^5+X_3^5+X_4^5+X_5^5-5\psi\,X_1X_2X_3X_4X_5=0
	$$
	in $\P^4$. For a fixed $\psi$, the equation admits the action of the discrete group
	$$
	G=\{(\zeta_5^{a_1},\dots, \zeta_5^{a_5}):a_1+\dots + a_5\equiv 0 \mod 5\}\cong (\Z/5\Z)^4
	$$
	via the map $(X_1,\dots, X_5)\mapsto(\zeta_5^{a_1}X_1,\dots,\zeta_5^{a_5}X_5)$,
	where $\zeta_5=e^{2\pi i/5}$ is the primitive $5$-th root of unity.
	Its mirror threefold is constructed from the orbifold $V(\psi)/G$.
	One way to realize the quotient is letting $y_j=X_j^5$ for $j=1,\dots,5$, $x_1=5\psi X_1\cdots X_5$ and $\l=\psi^{-5}$;
	the image is
	$$
	\hat{\mathcal V}(\l):=\hat{\mathcal V}_{\{\frac 15,\frac 25,\frac 35,\frac 45\}}(\l):\quad
	y_1+\dots+ y_5-x_1=0, \quad 5^{-5} \l x_1^5= y_1\cdots y_5.
	$$
	Resolving singularities, one gets a Calabi--Yau threefold
	$\overline{\hat{\mathcal V}(\l)}$ with generic $h^{2,1}$ equal to~$1$
	(see \cite{Candelas,Cox15} for details).
	By~\cite{Candelas}, the Picard--Fuchs differential operator of this mirror Calabi--Yau threefold is given by
	$$
	\theta^4-5^{-4}\lambda(5\theta+1)(5\theta+2)(5\theta+3)(5\theta+4),
	\quad \text{where}\; \theta:=\l \frac{d}{d\l},
	$$
	whose unique (up to scalar) holomorphic solution near zero is given by the hypergeometric function
	$$
	\sum_{k=0}^\infty\frac{(5k)!}{k!^5}\,(5^{-5}\l)^k
	=\PFQ{4}{3}{\frac 15 ,\, \frac 25 ,\, \frac 35 ,\, \frac 45}{1 ,\, 1 ,\, 1}{\l}.
	$$
	When $\l=1$, the corresponding Calabi--Yau threefold
	$\overline{\hat{\mathcal V}(1)}$ is defined over $\Q$ and
	it becomes rigid, that is, $h^{2,1}(\overline{\hat{\mathcal V}(1)})=0$
	meaning that its third Betti number $B_3=\dim H^3(\overline{\hat{\mathcal V}(1)},\C)$ is~2.
	It is shown by Schoen \cite{Schoen} that the $\ell$-adic Galois representation (of the absolute Galois group
	$G_\Q:=\operatorname{Gal}(\overline \Q/\Q)$) arising from \'etale cohomology
	$H^3_{\text{et}}(\overline{\hat{\mathcal V}(1)},\Ql)$ is \emph{modular} in the sense that it is isomorphic to
	the Galois representation attached to a weight~$4$ level $25$ Hecke eigenform $f=f_{\{\frac 15,\frac 25,\frac 35,\frac 45\}}$,
	labeled 25.4.a.b in the database \cite{LMFDB}.
	An expression of $f$ through the Dedekind eta-function
	$\eta(\tau)=q^{1/24}\prod_{n=1}^\infty(1-q^n)$, where $q=e^{2\pi i\tau}$,
	can be found in Table~\ref{tab1}; there we implement the standard notation $\eta_m=\eta_m(\tau)$ for $\eta(m\tau)$.
	As shown by Wan in \cite{Wan}, in the ordinary prime case both $V_\balpha(1)$ and its mirror  $\hat{\mathcal V}_\balpha(1)$ share the same unit root.
	
	Similar properties hold for the other one-parameter families $V_{\balpha}(\psi)$ listed in Tables~\ref{tab3} and~\ref{tab4}. Their corresponding mirror families $\hat{\mathcal V}_{\balpha}(\l)$ are provided in Table~\ref{table: tab5}.
	When $\l=1$, the corresponding mirror fiber $\overline{\hat{\mathcal V}_{\balpha}(1)}$ happens to be a rigid Calabi--Yau threefold defined over $\Q$,
	hence its $L$-function computed from the third \'etale cohomology group is also modular by \cite{Dieu,GY} (see Theorem~\ref{thm:mod} below).
	In what follows,
	$$
	f_{\balpha}(\tau)=\sum_{n=1}^\infty a_nq^n,
	\quad\text{where}\; a_1=1,
	$$
	denotes the weight 4 (normalized) modular form corresponding to
	$\overline{\hat{\mathcal V}_\balpha(1)}$.
	
	Based on some numerical evidence, Rodriguez-Villegas came up in \cite{RV} with a simple recipe to compute
	the $p$-th Fourier coefficients $a_p=a_p(f_{\balpha})$ (therefore, all of them) for these fourteen rigid Calabi--Yau threefolds.
	Proofs of his conjectural (super)congruences are the main result of the paper.
	
	\begin{Theorem}\label{thm:main}
		Let $r_1,r_2\in\bigl\{\frac 12,\frac13,\frac14,\frac16\bigr\}$
		or $(r_1,r_2)\in \bigl\{\bigl(\frac 15,\frac 25\bigr),\bigl(\frac 18,\frac 38\bigr), \bigl(\frac 1{10},\frac 3{10}\bigr), \bigl(\frac 1{12},\frac 5{12}\bigr)\bigr\}$.
		Then for each prime $p>5$, the truncated hypergeometric series
		\begin{equation}
		\PFQ{4}{3}{r_1 ,\, 1-r_1 ,\, r_2 ,\, 1-r_2 }{\quad 1 ,\quad 1 ,\quad 1}{1}_{p-1}
		=\sum_{k=0}^{p-1}\frac{(r_1)_k(1-r_1)_k(r_2)_k(1-r_2)_k}{k!^4}
		\label{eq:4F3trunc}
		\end{equation}
		satisfies
		\begin{equation}
		\PFQ{4}{3}{r_1 ,\, 1-r_1 ,\, r_2 ,\, 1-r_2 }{\quad 1 ,\quad 1 ,\quad 1}{1}_{p-1}
		\equiv a_p(f_{\{r_1,1-r_1,r_2,1-r_2\}}) \mod p^3.
		\label{eq:RV-1}
		\end{equation}
	\end{Theorem}
	
	The cases $(r_1,r_2)=(\frac 12,\frac 12)$ and $(\frac 15,\frac 25)$ have been obtained earlier by Kilbourn \cite{Kilbourn} and McCarthy \cite{McCarthy-RV}, respectively.
	Furthermore, the reduction of case $(r_1,r_2)=(\frac 12,\frac 14)$ to Kilbourn's result in \cite{Kilbourn} has been performed by McCarthy and Fuselier \cite{FM}.
	The remaining eleven cases are new. Our methods provide uniform proofs of all the fourteen hypergeometric cases.
	
	Notice that the Fourier coefficients of $f_{\balpha}$ can be related to the point counting on the corresponding variety over finite fields.
	A general strategy for this is set up in the recent paper \cite{BCM} by Beukers, Cohen and Mellit.
	In particular, the point counting on
	$\overline{\hat{\mathcal V}_{\{\frac 15,\frac 25,\frac 35,\frac 45\}}(\l)}$ over any finite field of characteristic
	different from~5 is given explicitly in terms of hypergeometric functions over finite fields, also known as finite hypergeometric functions.
	
	A unifying perspective underneath both classical and finite hypergeometric functions is absorbed by the notion of hypergeometric motives
	\cite{RVRW, RV2, Watkins}; see  \cite{DPVZ, H-T, LPSX, Long20, LLT} for some recent progress and further development of these theme. 
	Postponing details about them to Section~\ref{ss:charactersums}, we only indicate that the related hypergeometric motives
	in our settings are attached to the hypergeometric data
	\begin{equation}
	\l=1,\quad \balpha= \{r_1,1-r_1,r_2,1-r_2\},\quad \bbeta=\{1,1,1,1\},
	\label{hyp:main}
	\end{equation}
	where either $r_1,r_2\in\bigl\{\frac 12,\frac13,\frac14,\frac16\bigr\}$
	or $(r_1,r_2)\in \bigl\{\bigl(\frac 15,\frac 25\bigr),\bigl(\frac 18,\frac 38\bigr), \bigl(\frac 1{10},\frac 3{10}\bigr), \bigl(\frac 1{12},\frac 5{12}\bigr)\bigr\}$.
	These choices guarantee that the corresponding motives are all defined over $\Q$.
	For any prime $p>5$, the corresponding character sum $H_p(\balpha,\bbeta;1)$ is explicitly defined in~\eqref{eq:H2} below.
	
	Our next result gives an alternative expression for the Fourier coefficients of Hecke eigenform $f_{\balpha}$ in terms of $H_p(\balpha,\bbeta;1)$ and the information about its level.
	
	\begin{Theorem}
		\label{thm:level}
		Let $p>5$ be a prime and $\balpha$ and $\bbeta$  as above. Then the following equality holds\textup:
		\begin{equation}
		H_p(\balpha,\bbeta;1)=a_p(f_{\balpha})+\chi_{\balpha}(p)\cdot p,
		\label{eq:Haa}
		\end{equation}
		where $a_p(f_{\balpha})$ is the $p$-th coefficient of the normalized Hecke eigenform and  $\chi_{\balpha}$ is a Dirichlet character of order at most~$2$, whose precise description is given in Section~\textup{\ref{ss:char}}.
		The levels of $f_{\balpha}$ and the characters $\chi_\balpha$ are listed in Table~\textup{\ref{tab1}}
		\textup(and in Lemma~\textup{\ref{lem:char}} below\textup).
	\end{Theorem}
	
	In Table~\ref{tab1}, the notation $\chi_d$ in the final column stands for the quadratic character $\chi_d(p)=\bigl(\frac dp\bigr)$,
	the Legendre symbol, for unramified $p$; in particular, $\chi_1$ denotes the trivial character.
	In addition, Table~\ref{tab1} records the instances of known formulas for the eigenforms in the second column.
	
	\begin{table}[h]
		\caption{The Hecke eigenforms for rigid hypergeometric Calabi--Yau threefolds}
		\begin{center}
			\begin{tabular}{|c|c||l|c|c|}
				\hline
				$(r_1,r_2)$ & $f_{\boldsymbol\alpha}(\tau)$ &\quad\;\;level&LMFDB label&$\chi_{\boldsymbol\alpha}$\\ \hline\hline
				$(\tfrac12,\tfrac12)$ & $\eta_2^4\eta_4^4$ & $\phantom{00}8=2^3\vphantom{|^{0^0}}$ &
				\href{http://www.lmfdb.org/ModularForm/GL2/Q/holomorphic/8/4/a/a/}{\texttt{8.4.a.a}}
				& $\chi_1$ \\[.5mm] \hline\vphantom{$\big|^0$}%
				$(\tfrac12,\tfrac13)$ & $\eta_6^{14}/(\eta_2^3\eta_{18}^3)-3\eta_2^3\eta_6^2\eta_{18}^3$ & $\phantom036=2^2\cdot3^2$ &
				\href{http://www.lmfdb.org/ModularForm/GL2/Q/holomorphic/36/4/a/a/}{\texttt{36.4.a.a}}
				& $\chi_{3}$ \\[.5mm] \hline\vphantom{$\big|^0$}%
				$(\tfrac12,\tfrac14)$ & $\eta_4^{16}/(\eta_2^4\eta_8^4)$ & $\phantom016=2^4$ &
				\href{http://www.lmfdb.org/ModularForm/GL2/Q/holomorphic/16/4/a/a/}{\texttt{16.4.a.a}}
				& $\chi_{2}$ \\[.5mm] \hline\vphantom{$\big|^0$}%
				$(\tfrac12,\tfrac16)$ & $\begin{aligned} \vphantom{\big|^0}
				& \eta_4^3\eta_6^6\eta_{18}^2/(\eta_{12}^2\eta_{36}) - 3\eta_2^2\eta_6^6\eta_{36}^3/(\eta_4\eta_{12}^2) \\[-1mm] &\quad + 8\eta_2^3\eta_{12}^6\eta_{36}^2/(\eta_6^2\eta_{18}) - 16\eta_{12}^{12}/\eta_6^4
				\end{aligned}$ & $\phantom072=2^3\cdot3^2$ &
				\href{http://www.lmfdb.org/ModularForm/GL2/Q/holomorphic/72/4/a/b/}{\texttt{72.4.a.b}}
				& $\chi_1$ \\[.5mm] \hline\vphantom{$\big|^0$}%
				$(\tfrac13,\tfrac13)$ & $\eta_1^3\eta_3^4\eta_9-27\eta_3\eta_9^4\eta_{27}^3$ & $\phantom027=3^3$ &
				\href{http://www.lmfdb.org/ModularForm/GL2/Q/holomorphic/27/4/a/a/}{\texttt{27.4.a.a}}
				& $\chi_1$ \\[.5mm] \hline\vphantom{$\big|^0$}%
				$(\tfrac13,\tfrac14)$ & $\eta_3^8$ & $\phantom{00}9=3^2$ &
				\href{http://www.lmfdb.org/ModularForm/GL2/Q/holomorphic/9/4/a/a/}{\texttt{9.4.a.a}}
				& $\chi_{6}$ \\[.5mm] \hline\vphantom{$\big|^0$}%
				$(\tfrac13,\tfrac16)$ & $\eta_6^{10}/\eta_{18}^2 - 27\eta_{18}^{10}/\eta_6^2 + 9\eta_6^7\eta_{54}^3/\eta_{18}^2 - 9\eta_2^3\eta_{18}^7/\eta_6^2$ & $108=2^2\cdot3^3$ &
				\href{http://www.lmfdb.org/ModularForm/GL2/Q/holomorphic/108/4/a/a/}{\texttt{108.4.a.a}}
				& $\chi_{3}$ \\[.5mm] \hline\vphantom{$\big|^0$}%
				$(\tfrac14,\tfrac14)$ & $\eta_4^{10}/\eta_8^2-8\eta_8^{10}/\eta_4^2$ & $\phantom032=2^5$ &
				\href{http://www.lmfdb.org/ModularForm/GL2/Q/holomorphic/32/4/a/a/}{\texttt{32.4.a.a}}
				& $\chi_1$ \\[.5mm] \hline\vphantom{$\big|^0$}%
				$(\tfrac14,\tfrac16)$ & $\eta_{12}^{32}/(\eta_6^{12}\eta_{24}^{12})+16\eta_6^4\eta_{24}^4$ & $144=2^4\cdot3^2$ &
				\href{http://www.lmfdb.org/ModularForm/GL2/Q/holomorphic/144/4/a/f/}{\texttt{144.4.a.f}}
				& $\chi_{2}$ \\[.5mm] \hline\vphantom{$\big|^0$}%
				$(\tfrac16,\tfrac16)$ &    & $216=2^3\cdot3^3$ &
				\href{http://www.lmfdb.org/ModularForm/GL2/Q/holomorphic/216/4/a/c/}{\texttt{216.4.a.c}}
				& $\chi_1$  \\[.5mm] \hline\vphantom{$\big|^0$}%
				$(\tfrac15,\tfrac25)$ & $\eta_5^{10}/(\eta_1\eta_{25})+5\eta_1^2\eta_5^4\eta_{25}^2$ & $\phantom025=5^2$ &
				\href{http://www.lmfdb.org/ModularForm/GL2/Q/holomorphic/25/4/a/b/}{\texttt{25.4.a.b}}
				& $\chi_{5}$ \\[.5mm] \hline\vphantom{$\big|^0$}%
				$(\tfrac18,\tfrac38)$ & $\begin{aligned} \vphantom{\big|^0}
				& \eta_1^2\eta_2\eta_4^3\eta_8^3/\eta_{16}
				+2\eta_2^2\eta_4^3\eta_8^2\eta_{16}
				+8\eta_4^2\eta_8^2\eta_{16}^2\eta_{32}^2
				- 24\eta_4^2\eta_{16}^2\eta_{32}^4
				\\[-1mm] &\quad
				-16\eta_2^2\eta_4\eta_8^2\eta_{64}^4/\eta_{16}
				-64\eta_8\eta_{16}^2\eta_{32}^3\eta_{64}^2
				+32\eta_{16}^2\eta_{32}^5\eta_{64}^2/\eta_8
				\\[-1mm] &\quad
				-32\eta_4^4\eta_8\eta_{32}\eta_{64}\eta_{128}^2/\eta_{16}
				-64\eta_8\eta_{16}^3\eta_{32}\eta_{64}\eta_{128}^2
				\\[-1mm] &\quad
				-256\eta_{16}^3\eta_{32}^3\eta_{64}\eta_{128}^2/\eta_8
				+128\eta_4^2\eta_{32}^2\eta_{64}^3\eta_{128}^2/\eta_{16}
				\\[-1mm] &\quad
				-256\eta_4^4\eta_{32}\eta_{128}^4/\eta_8
				-128\eta_{16}^4\eta_{32}\eta_{128}^4/\eta_8
				-512\eta_4^2\eta_{64}^2\eta_{128}^4
				\end{aligned}$ & $128=2^7$ &
				\href{http://www.lmfdb.org/ModularForm/GL2/Q/holomorphic/128/4/a/b/}{\texttt{128.4.a.b}}
				& $\chi_{2}$ \\[.5mm] \hline\vphantom{$\big|^0$}%
				$(\tfrac1{10},\tfrac3{10})$ & & $200=2^3\cdot5^2$ &
				\href{http://www.lmfdb.org/ModularForm/GL2/Q/holomorphic/200/4/a/f/}{\texttt{200.4.a.f}}
				& $\chi_1$ \\[.5mm] \hline\vphantom{$\big|^0$}%
				$(\tfrac1{12},\tfrac5{12})$ & & $864=2^5\cdot3^3$ &
				\href{http://www.lmfdb.org/ModularForm/GL2/Q/holomorphic/864/4/a/a/}{\texttt{864.4.a.a}}
				& $\chi_1$\\[.5mm] \hline
			\end{tabular}
		\end{center}
		\label{tab1}
	\end{table}
	
	The supercongruences in Theorem~\ref{thm:main} are of interest not only for their own sake; they also happen to encode remarkable symmetries of underlying varieties.
	In terms of applications, the symmetries can be used to determine the $L$-functions of hypergeometric motives more efficiently.
	
	The paper is organized as follows. In Section~\ref{sec2},
	we briefly go through historical developments of the supercongruences.
	\cy	Section~\ref{ss:Dwork} is for the ordinary case, namely, when the finite sum in \eqref{eq:4F3trunc} does not vanish modulo a prime $p>5$. In this case we show,  in Theorem \ref{thm:3}, that the truncated hypergeometric series and Dwork $p$-adic unit root agree modulo $p^3$.  This is based on Dwork's method and the $p$-adic perturbation method. 
	En route, a hypergeometric machinery is used to design auxiliary identities\,---\,a machinery that naturally complements the $p$-adic perturbation method from~\cite{LR}.
	We believe this approach is new. Then, in Section~\ref{ss:charactersums}, hypergeometric motives  and  finite hypergeometric functions are introduced.
	Section~\ref{ss:CY} discuses the fourteen one-parameter families of
	Calabi--Yau threefolds as well as the associated mirror Calabi--Yau
	threefolds and corresponding rigid Calabi--Yau threefolds.
	This section concludes with the proof of Theorem~\ref{thm:level}.
	Then our proof of Theorem~\ref{thm:main} for all primes $ p>5$ is presented in Section~\ref{ss:charsum} building on the results from the previous sections.
	Importantly, the hypergeometric machinery in Section~\ref{ss:Dwork} also gives us a good control of the $p$-adic error terms here.
	Finally, in Section~\ref{sec:7} we review \cy our two strategies of proving the supercongruences and indicate some future potential development of the methods.  The relation between the Dwork unit roots and varieties $V_\balpha(1)$ listed in Table \ref{table: tab5} is given in Appendix~\ref{ss:A} via 1-dimensional commutative formal group laws (1-CFGL).  In part~\ref{ss:A.2} we link the cohomology groups of singular
	Calabi--Yau threefolds (at the conifold point) with their smooth models via Stienstra's work \cite{Stienstra,Stienstra-Formal}. These relations present interconnections among the $p$-adic unit roots.
	We demonstrate them at length on one particular example $(r_1,r_2)=(\frac 12,\frac12)$ which corresponds to the complete intersection of \emph{four} hypersurfaces in a weighted projective space. The other cases are characterized as intersections of two hypersurfaces or as a single hypersurface;
	they can be computed in a similar manner, and we limit ourselves to listing the logarithms of the corresponding 1-CFGL in Table~\ref{tab6}.   \bk

	We would like to point out that McCarthy previously attempted \cite{McCarthy-extending} to prove supercongruences \eqref{eq:RV-1}, in their full generality,
	by defining a class of $p$-adic hypergeometric functions which he called $G$-functions.

	\section{Review of related background}
	\label{sec2}
	
	\subsection{Supercongruences}
	\label{sec2.1}
	In the notation of Section~\ref{sec1}, let $\overline{\hat{\mathcal V}(1)}$ denote the quintic rigid Calabi--Yau threefold
	and $f=f_{\{\frac 15,\frac 25,\frac 35,\frac 45\}}$ the weight $4$ modular form associated with it.
	One way to see the relation of the $p$-th Fourier coefficient $a_p(f)$ to a truncated hypergeometric sum
	comes from Stienstra's results in \cite{Stienstra}. We know from this work that there is a one-dimensional formal group arising from $H^3(V(1),\hat {G}_{m,V(1)})$ whose logarithm is given by
	\begin{equation}
	\sum_{n\ge 1} \frac{A_n}n \tau^n
	=\sum_{n=1}^\infty\biggl(\sum_{k\ge0}\binom{n-1}{5k}\frac{(5k)!}{k!^5}\,(-5)^{n-1-5k}\biggr)\frac{\tau^n}n;
	\label{Antau}
	\end{equation}
	see \cite[Theorem 1]{Stienstra} for the formula for $A_n$.
	Similar conclusions can be drawn for other thirteen cases using the homogeneous equations listed in Tables~\ref{tab3} and~\ref{tab4};  the recipe is given explicitly in the recent paper \cite{BV} by Beukers and Vlasenko, which is an extension of \cite{Stienstra}. \bk
	By the modularity of $\overline{\mathcal V(1)}$, this formal group is isomorphic to the formal group whose logarithm is defined by
	$$
	\sum_{n=1}^\infty a_n(f)\frac{\tau^n}n.
	$$
	In particular, when $n=p$ is a prime different from $5$,
	$$
	A_p=\sum_{k\ge 0}\binom{p-1}{5k}\frac{(5k)!}{k!^5}\,(-5)^{p-1-5k}\equiv \sum_{k=0}^{p-1}\frac{(5k)!}{k!^5}\,5^{-5k}\mod p.
	$$

	From the standard results in commutative formal group laws, it then follows that for each prime $p>5$,
	\begin{equation*}
	\PFQ{4}{3}{\frac15,\,\frac25,\,\frac35,\,\frac45}{1,\,1,\,1}{1}_{p-1}
	=\sum_{k=0}^{p-1}\frac{(5k)!}{k!^5}\,5^{-5k}
	\equiv a_p(f) \mod p.
	\end{equation*}
	Rodriguez-Villegas later noticed and conjectured in \cite{RV} that this congruence
	(and, similarly, the remaining thirteen ones for rigid hypergeometric Calabi--Yau threefolds) hold true modulo $p^3$;
	these are precisely the supercongruences \eqref{eq:RV-1} in Theorem~\ref{thm:main}.

	The more recent work \cite{RRW} of Roberts and Rodriguez-Villegas brings to life
	refined predictions related to \eqref{eq:RV-1};
	it also indicates a heuristics underlying the supercongruences that explains the exponent~3 of the modulus
	by the Hodge filtration of the corresponding hypergeometric motives.
	One part of the story is Dwork's unit-root theory in \cite{Dwork-padic}, in which the truncated hypergeometric series 
	\begin{equation}
	F_s=F_s(\balpha)
	:=\PFQ{4}{3}{r_1,\,1-r_1,\,r_2,\,1-r_2}{\quad 1,\quad 1,\quad 1}{1}_{p^s-1}
	\quad\text{for}\; s\ge1
	\label{trunc}
	\end{equation}
	plays an important role. Following Dwork, if $F_1\not\equiv 0\mod p$, then there exists $\gamma_p\in \Z_p^\times$ such that  \begin{equation}\label{eq:gp}
	F_s/F_{s-1}\equiv \gamma_p \mod{p^s}
	\end{equation} for $s\ge1$.    The claim $F_{s+1}/F_s\equiv F_1\mod{p^3}$ for $s=1,2$ (hence for all $s\ge0$)\,---\,Theorem~\ref{thm:3} below\,---\,and its proof
	form the contents of Section~\ref{ss:Dwork}.  Later in \cite{Dwork-padic}, Dwork showed how such unit roots can be used to compute the zeta functions of the fibers on the example of one-parameter deformation of Fermat hypersurfaces.  In \cite[Theorem 4.3]{Yu},  Yu generalized Dwork's result to Dwork families of Calabi--Yau manifolds, including the quintic family from the Introduction. For the fiber $\psi=1$ of the family,    this means that if $F_1(\{\frac15,\frac25,\frac35,\frac45\})\not\equiv 0 \mod p$, then $\gamma_p$ defined in \eqref{eq:gp} happens to be a reciprocal root of the zeta function of $V(1)$ which is invertible in $\Z_p$. Using the construction of the quintic mirror, it is known that it is also the reciprocal root of the zeta function of  $\overline{\hat{\mathcal V}_{\balpha}(1)}$ which is invertible in $\Z_p$ (see \cite{CdOR} by Candelas, de la Ossa and Rodriguez-Villegas).
	
	\subsection{$p$-adic perturbation method}
	Among a variety of techniques proving supercongruences for truncated hypergeometric sums, the most relevant one is the so-called $p$-adic perturbation method
	described in \cite{LR} and originated in \cite{CLZ,Long}. It is efficient in dealing with entities that can be represented through Gamma values.
	In plain  language,  the method breaks down the entities into two parts, one in terms of the $p$-adic Gamma functions $\G_p(x)$ that have good local analytic property
	and the other one that collects all $p$-powers. (We will  use this strategy in our later discussion.)
	As a consequence, we can reduce a desired supercongruence to a major term and error terms.
	In the final stage, eliminations of the error terms are often done through known hypergeometric identities, which are perturbed $p$-adically using additional inputs
	like Galois symmetries. This approach was later used in \cite{WIN3b,Swisher}.
	
	\subsection{Hidden symmetries}\label{ss:HS}
	Supercongruences often seem to highlight some hidden symmetries typically appearing as classical hypergeometric identities,
	symmetries that are crucial to efficiency of the $p$-adic perturbation method.
	For example, Ahlgren and Ono \cite{AO00} used a ``hidden symmetry'' identity
	\begin{equation}\label{eq:Fn}
	\sum_{k=1}^n \binom{n+k}{k}^2\binom{n}{k}^2(1+2kH_{n+k}+2kH_{n-k}-4kH_k)=0
	\end{equation}
	to establish the following supercongruence of the Ap\'ery numbers: for any prime $p>2$,
	\begin{equation}\label{eq:cc}
	A\biggl(\frac{p-1}2\biggr)
        :=\PFQ{4}{3}{\frac{1+p}2,\,\frac{1+p}2,\,\frac{1-p}2,\,\frac{1-p}2}{\quad 1,\quad 1,\quad 1}{1}
	\equiv a_p(f_{\{\frac 12,\frac 12,\frac 12,\frac 12\}}) \mod p^2.
	\end{equation}
	Here $f_{\{\frac 12,\frac 12,\frac 12,\frac 12\}}(\tau)=\eta(2\tau)^4\eta(4\tau)^4$ (as in Table~\ref{tab1}) and
	$$
	H_k:=\sum_{j=1}^k\frac 1j
	$$
	denotes the $k$-th partial harmonic sum, with $H_0=0$.
	As shown in \cite{Beukers}, the formal group method leads to the congruence \eqref{eq:cc} modulo $p$ only;
	identity \eqref{eq:Fn} extends it modulo~$p^2$.
	This latter identity \eqref{eq:Fn} was verified by a clever execution of the Wilf--Zeiberger method of creative telescoping,
	though, in fact, it is a rather elementary analysis of the residue sum for the rational function
	\begin{equation}\label{eq:ratfun}
	\frac{\prod_{j=1}^n(t-j)^2}{\prod_{j=0}^n(t+j)^2}
	\end{equation}
	(see \cite{Zudilin} and also \cite{McCarthy-identity}, \cite[Lemma 5.1]{OSZ}).
	Extending the work \cite{AO00}, Kilbourn \cite{Kilbourn} demonstrated that, for primes $p>2$,
	\begin{equation}\label{eq:Kilbourn}
		\PFQ{4}{3}{\frac 12 ,\, \frac 12 ,\, \frac 12 ,\, \frac 12}{1 ,\, 1 ,\, 1}{1}_{p-1}
	\equiv a_p(f_{\{\frac 12,\frac 12,\frac 12,\frac 12\}}) \mod p^3,
	\end{equation}
	a supercongruence originally conjectured by van Hamme in \cite{vH}.
	The later development of the method towards proving some other instances of \eqref{eq:RV-1} was undertaken in \cite{FM, McCarthy-RV}.
	
	Quite remarkably, a somewhat simpler companion identity associated to \eqref{eq:Fn} exists,
	\begin{equation}\label{eq:Fn2}
	\sum_{k=0}^n \binom{n+k}{k}^2\binom{n}{k}^2(H_{n+k}+H_{n-k}-2H_k)=0,
	\end{equation}
	that possesses the same (if not simpler) proof \cite{Zudilin}. With the help of \eqref{eq:Fn2},
	we can deduce the supercongruence \eqref{eq:Kilbourn} for ordinary primes $p$ without using a heavy machinery
	of character sums\,---\,somewhat more straightforward than \cite{Kilbourn, McCarthy-RV}.
	The next section details the related approach and gives a generalization of \eqref{eq:Fn2},
	which on addition of $p$-adic perturbation terms is capable of treating all the cases in Theorem~\ref{thm:main} for ordinary primes.

	\section{Dwork's approach for ordinary primes}\label{ss:Dwork}
	
	\subsection{$p$-adic setup}\label{ss:4.1}
	Dwork \cite{Dwork-padic} laid down a framework for $p$-adic hypergeometric functions, which we discuss here in relation to the objectives of this paper. \cy In particular, we show in Theorem~\ref{thm:3} below that for any ordinary good prime~$p$
	the finite sum \eqref{eq:4F3trunc} agrees with the corresponding Dwork unit root modulo~$p^3$. \bk
	
	We use $\lf x\rf$ for the floor function of $x\in \mathbb R$ and denote $\{x\}:=x-\lf x\rf$ the fractional part.
	For the discussion in this section, we work over the ring of $p$-adic integers $\Z_p$ with $p>5$ being any fixed prime;
	see \cite{Cohen} for more detailed information on $\Z_p$ and the $p$-adic Gamma function $\G_p(x)$.
	Furthermore, for $r\in \Z_p$, let $[r]_0$ denote its first $p$-adic digit.
	
	\begin{Lemma}
		Given an integer $k$, $0\le k<p$, and $r\in \Z_p^\times$, the rising factorial $(r)_k$ is in $\Z_p^\times$ if and only if $k\le[-r]_0$.
	\end{Lemma}
	
	\begin{proof}
		This follows from $(r)_k=r(r+1)\cdots(r+k-1)$ and the definition of $[-r]_0$.
	\end{proof}
	
	The following discussion is based on Dwork's work on $p$-adic cycles \cite{Dwork-padic}.
	We will use Dwork's dash operation\,---\,the map $\Q\cap\Z_p \to \Q\cap\Z_p$ defined by
	$$
	r'= (r+[-r]_0)/p,
	$$
	which despite of its notational appearance has nothing to do with the usual derivative or derivations.
	It is easy to see that $pr'-r=[-r]_0 \in\{0,1,\dots,p-1\}$.
	If we write $r=\ell/d$ with $\gcd(\ell,d)=1$, the $p$-adic integer $r'$ is $\tilde{\ell}/d$ with $\tilde{\ell}\equiv \ell/p \mod d$.
	For each multi-set $\balpha=\{r_1,r_2,r_3,r_4\}=\{r_1,r_2,1-r_1,1-r_2\}$ in Table~\ref{tab1} and a prime $p$ not dividing the common denominator of $r_1,r_2$, the multi-set $\{r_1',r_2',r_3',r_4'\}$ is just a permutation of $\{r_1,r_2,r_3,r_4\}$ (we will cast this property of $\balpha$ as the closedness under the Galois conjugation in Section~\ref{ss:charactersums}).
	In particular, the dash operation preserves any of our fourteen multi-sets $\balpha$ for every prime $p > 5$.
	
	\begin{Proposition}[Dwork]
		\label{prop:Dwork}
		Given a prime $p$ and a multi-set $\balpha=\{r_1,r_2,r_3,r_4\}$ with $r_1,r_2\in(0,1)\cap\Z_p^\times$ such that $\balpha$ is preserved by the dash operation,
		for any integers $s\ge t\ge 1$ we have $F_sF_{t-1}\equiv F_tF_{s-1}\mod p^t$, where $F_s=F_s(\balpha)$ is defined in \eqref{trunc}.
		In particular, if $F_1(\balpha)\not\equiv0\mod p$, then there exists a unique  $\gamma_p=\gamma_p(\balpha)\in \Z_p^\times$ such that for any integer $s\ge 1$,
		$$
		F_s/F_{s-1}\equiv \gamma_p\mod p^s.
		$$
	\end{Proposition}
	
	\begin{proof}
		Here we indicate how to deduce the proposition from the results in \cite{Dwork-padic}.
		Denote $A(n)=\bigl(\prod_{i=1}^4 (r_i)_n\bigr)/n!^4$.
		Since $\balpha$ is closed under the dash operator, by \cite[Corollary 2]{Dwork-padic} we deduce that
		$A(n)/A(\lf n/p\rf)\in \Z_p$ and for all integers $m,s\ge 1$,
		$$
		\frac{A(n)}{A(\lf n/p\rf)}\equiv \frac{A(n+mp^{s+1})}{A(\lf n/p\rf+mp^s)}
		\mod p^{s+1}.
		$$
		This property and \cite[Theorems 2, 3]{Dwork-padic} imply that $B^{(i)}(n)=A(n)$ for all $i=1,2,\dots$,
		and \cite[Equation (3.2$'$)]{Dwork-padic} at $x=1$ becomes equivalent to the claim $F_sF_{t-1}\equiv F_tF_{s-1}\mod p^t$ for integers $s\ge t\ge 1$. This implies the existense of the $p$-adic limit $\gamma_p$ of $F_s/F_{s-1}$ as $s\to\infty$ when $F_1\not\equiv0\pmod p$; in other words, when $p$ is an ordinary prime. The limit $\gamma_p$ is a unit in $\Z_p$, because $\gamma_p\equiv F_1\pmod p$.
	\end{proof}

	\begin{Theorem}\label{thm:3}
		Let $\balpha=\{r_1,r_2,1-r_1,1-r_2\}$ be one of the fourteen multi-sets from Table~\textup{\ref{tab1}} and $p$ a prime  such that $r_1,r_2\in \Z_p^\times$. Then for any integer $s\ge 1$,
		\begin{equation*}
		F_{s+1}(\balpha)\equiv F_s(\balpha)F_1(\balpha)\mod p^3.
		\end{equation*}
	\end{Theorem}
	
	We will require some preparation before proving Theorem \ref{thm:3}. The following reduction of a quotient of rising factorials makes use
	of the relation between the Gamma and $p$-adic Gamma functions: for prime $p>2$ and a positive integer~$n$,
	\begin{equation}
	\label{gammas}
	\G(n)=(-1)^n \G_p(n)\,\biggl\lfloor \frac {n-1}p\biggr\rfloor!\,p^{\lfloor(n-1)/p\rfloor}.
	\end{equation}
	
	\begin{Lemma}
		\label{Dw-lemma}
		Let $k\in \Z_{\ge 0}$, $a=[k]_0$ and $b=(k-a)/p$, that is, $k=a+bp$.  Then for any  $r\in\Z_p^\times$,
		$$
		\frac{(r)_k}{(1)_k}
		=\frac{-\Gamma_p(r+k)}{\G_p(1+k)\,\G_p(r)}\frac{(r')_b}{(1)_b}\cdot \bigl((r'+b)p\bigr)^{\nu(a,[-r]_0)},
		$$
		where
		\begin{equation}
		\label{eq:nu}
		\nu(a,x)=-\biggl\lfloor\frac{x-a}{p-1}\biggr\rfloor
		=\begin{cases} 0 & \text{if }\;  a \le x, \\1 & \text{if} \;  x<a<p.
		\end{cases}
		\end{equation}
	\end{Lemma}
	
	\begin{proof}
		Write $r \in \Z_p^\times$ as $r=[r]_0+ph$. Then $[-r]_0=p-[r]_0$ and $r'=h+1$. First, assume that $r$ is a positive integer.
		Then from \eqref{gammas} we deduce that
		$$
		\frac{(r)_k}{(1)_k}
		=\frac{\Gamma(r+k)}{\G(1+k)\,\G(r)}
		=\frac{-\Gamma_p(r+k)}{\G_p(1+k)\,\G_p(r)}\,\frac{\lfloor(r+k-1)/p\rfloor!}{\lfloor k/p\rfloor!\,\lfloor(r-1)/p\rfloor!}
		\,p^{\lfloor(r+k-1)/p\rfloor-\lfloor k/p\rfloor-\lfloor(r-1)/p\rfloor}.
		$$
		Since $r+k =a+[r]_0+p(h+b)$,
		we have
		$$
		\biggl\lfloor \frac {r+k-1}p\biggr\rfloor
		=h+b+
		\begin{cases}
		0& \text{if}\; a\le p-[r]_0=[-r]_0,\\
		1& \text{if}\; [-r]_0< a<p,
		\end{cases}
		$$
		so that for $0<a\le[-r]_0$,
		$$
		\frac{\lfloor(r+k-1)/p\rfloor!}{\lfloor k/p\rfloor!\,\lfloor(r-1)/p\rfloor!}\,p^{\lfloor(r+k-1)/p\rfloor-\lfloor k/p\rfloor-\lfloor(r-1)/p\rfloor}
		=\frac{(h+b)!}{b!h!}=\frac{(h+1)_b}{(1)_b}=\frac{(r')_b}{(1)_b},
		$$
		and for $[-r]_0< a<p$,
		$$
		\frac{\lfloor(r+k-1)/p\rfloor!}{\lfloor k/p\rfloor!\,\lfloor(r-1)/p\rfloor!}\,p^{\lfloor(r+k-1)/p\rfloor-\lfloor k/p\rfloor-\lfloor(r-1)/p\rfloor}
		=\frac{(h+b)!}{b!h!}\,(h+b+1)p=\frac{(r')_b}{(1)_b}\,(r'+b)p.
		$$
		Therefore,
		$$
		\frac{(r)_k}{(1)_k}
		=\frac{-\Gamma_p(r+k)}{\G_p(1+k)\,\G_p(r)}\,\frac{(r')_b}{(1)_b}\,\bigl((r'+b)p\bigr)^{\nu(a,[-r]_0)}.
		$$
		The identity holds for all $r\in \Z_p^\times$ by the continuity of the $p$-adic Gamma function.
	\end{proof}
	
	We complement Lemma~\ref{Dw-lemma} by its particular instance $k=a$ (so that $b=0$):
	\begin{equation}\label{eq:GammaRis}
	\frac{-\Gamma_p(r+a)}{\Gamma_p(1+a)\,\Gamma_p(r)}
	=\frac{(r)_a}{a!}\cdot\frac1{(r'p)^{\nu(a,[-r]_0)}},
	\end{equation}
	which in turn implies the following for $k=a+bp$:
	\begin{align}
	\frac{(r)_{a+bp}}{(1)_{a+bp}}
	&=\frac{-\Gamma_p(r+a)}{\Gamma_p(1+a)\,\Gamma_p(r)}\,\frac{(r')_b}{(1)_b}((r'+b)p)^{\nu(a,[-r]_0)}
	\frac{\G_p((r+a)+bp)\G_p(1+a)}{\G_p(r+a)\G_p((1+a)+bp)}
	\nonumber\\
	&=\frac{(r)_a}{a!}\,\frac{(r')_b}{(1)_b}\biggl(1+\frac b{r'}\biggr)^{\nu(a,[-r]_0)}
	\frac{\G_p((r+a)+bp)\G_p(1+a)}{\G_p(r+a)\G_p((1+a)+bp)}.
	\label{eq:red}
	\end{align}
	This ``key reduction'' formula \eqref{eq:red} is instrumental in separating the $p$-adic terms in a way convenient to our future analysis.
	
	We will use the following local analytic properties of the $p$-adic Gamma function (see for example \cite{LR}).
	For $t,s\in \Z_p$,
	\begin{equation}
	\label{eq:pGamma}
	\G_p(t+sp)=\G_p(t)\biggl(1+sp\,G_1(t)+\frac{(sp)^2}{2}\,G_2(t)\biggr) \mod p^3,
	\end{equation}
	where$$
	G_1(t):=\frac{\d}{\d t}\log\G_p(t)=\frac{\d\G_p(t)/\d t}{\G_p(t)}
	\quad\text{and}\quad
	G_2(t):=\frac{\d^2\G_p(t)/\d t^2}{\G_p(t)}.
	$$
	From definition,
	$$
	\frac{\d}{\d t} G_1(t)=G_2(t)-G_1(t)^2.
	$$
	
	\begin{Lemma}
		\label{gp-poch}
		For any $t\in \Z_p$ and an integer $a\in\{0,1,\dots,p-1\}$, we have
		\begin{align*}
		\frac{\d}{\d t} (t)_a&=(t)_a\biggl(G_1(t+a)-G_1(t)+\frac{\nu(a,[-t]_0)}{t+[-t]_0}\biggr);
		\\
		\frac{\d^2}{\d t^2} (t)_a&=(t)_a\biggl(\biggl(G_1(t+a)-G_1(t)+\frac{\nu(a,[-t]_0)}{t+[-t]_0}\biggr)^2
		\\ &\qquad
		+G_2(t+a)-G_2(t)-G_1(t+a)^2+G_1(t)^2-\frac{\nu(a,[-t]_0)}{(t+[-t]_0)^2}\biggr),
		\end{align*}
		where $\nu(a,x)$ is defined in \eqref{eq:nu}. Notice that $t+[-t]_0=pt'$.
	\end{Lemma}
	
	\begin{proof}
		By the functional equation of the $p$-adic Gamma function, for $0\le a<p$ we have
		\begin{equation}
		(t)_a=t(t+1)\cdots(t+a-1)=(-1)^a\frac{\G_p(t+a)}{\G_p(t)}(t+[-t]_0)^{\nu(a,[-t]_0)}.
		\label{(t)_k}
		\end{equation}
		Logarithmically differentiating the equality we derive the required claims.
	\end{proof}
	
	Let $\balpha=\{r_1,r_2,r_3,r_4\}=\{r_1,r_2,1-r_1,1-r_2\}$ be one of the fourteen multi-sets from Table~\textup{\ref{tab1}} and $p$ a prime such that $r_1,r_2\in \Z_p^\times$.
	The dash operation preserves the multi-set; from now on we will numerate the entries in the multi-set $\{r_1,r_2,r_3,r_4\}$ in such a way that
	\begin{equation}\label{eq:order}r_1'\le r_2'\le r_3'\le r_4'.\end{equation}
	This inequality, the structure of the entries in $\{r_1,r_2,r_3,r_4\}=\{r_1',r_2',r_3',r_4'\}$ and the trivial property $(1-r)'=1-r'$ for rational $r\in(0,1)\cap\Z_p^\times$ result in
	$$
	r_1+r_4=r_2+r_3=1 \quad\text{and}\quad r_1'+r_4'=r_2'+r_3'=1.
	$$
	Furthermore, denote 
	$$
	a_j:=[-r_j]_0=pr_j'-r_j, \quad \text{for} \; j=1,2,3,4.
	$$
	From the ordering chosen in \eqref{eq:order}, if $i<j$ then  $a_i-a_j=p(r_i'-r_j')-(r_i-r_j)\le-(r_i-r_j)<1$, hence $a_i\le a_j$ as they are both integers. Putting together, 
	$$
	a_1\le a_2\le a_3\le a_4 \quad\text{and}\quad a_1+a_4=a_2+a_3=p-1.
	$$
	The extra factors appearing in \eqref{eq:red} are collected in the expression
	\begin{equation}\label{eq:Lambda}
	\Lambda_{\balpha}(a+bp):=\prod_{j=1}^4\biggl(1+\frac b{r_j'}\biggr)^{\nu(a,[-r_j]_0)}=\begin{cases}
	1 &\text{if}\; 0\le a\le a_1, \\
	(1+b/r_1') &\text{if}\; a_1<a\le a_2, \\
	(1+b/r_1')(1+b/r_2') &\text{if}\; a_2<a\le a_3 \\
	\end{cases}
	\end{equation}
	(we omit the other cases in view of their irrelevance)
	and, for $0\le a<p$, the $p$-adic order of the Pochhammer quotient
	\begin{equation}
	\frac{\prod_{j=1}^4(r_j)_a}{a!^4}
	\label{prodr}
	\end{equation}
	is equal to $s\in\{0,1,2,3,4\}$ if and only if $a_s<a\le a_{s+1}$, where we additionally set $a_0=-1$ and $a_5=p-1$.
	
	It follows from \eqref{eq:pGamma} that
	\begin{equation}\label{eq:J-role}
	\frac{\prod_{j=1}^4\G_p((r_j+a)+bp)}{\G_p((1+a)+bp)^4}
	\equiv\frac{\prod_{j=1}^4\G_p(r_j+a)}{\G_p(1+a)^4}\bigl(1+J_1(a)\,bp+J_2(a)\,(bp)^2\bigr) \mod p^3,
	\end{equation}
	where the coefficients $J_1(a)$ and $J_2(a)$ are given by
	\begin{equation}\label{J_12}
	\begin{aligned}
	J_1(a)=J_1(a,\balpha)
	&:=\sum_{j=1}^4\bigl(G_1(r_j+a)-G_1(1+a)\bigr),
	\\
	J_2(a)=J_2(a,\balpha)
	&:=10G_1(1+a)^2-4G_1(1+a)\sum_{j=1}^4G_1(r_j+a)
	\\ &\quad
	+\sum_{1\le j<\ell\le4}G_1(r_j+a)G_1(r_\ell+a)+\frac12\sum_{j=1}^4\bigl(G_2(r_j+a)-G_2(1+a)\bigr).
	\end{aligned}
	\end{equation}
	By the key reduction formula \eqref{eq:red} and from \eqref{eq:J-role},
	\begin{align}
	F_{s+1}(\balpha)
	&=\sum_{a=0}^{p-1}\sum_{b=0}^{p^s-1}\frac{\prod_{j=1}^4(r_j)_{a+bp}}{(1)_{a+bp}^4}
	\nonumber\\
	&\equiv\sum_{b=0}^{p^s-1}\frac{\prod_{j=1}^4(r_j')_b}{b!^4}
	\sum_{a=0}^{p-1}\frac{\prod_{j=1}^4(r_j)_a}{a!^4}
	\nonumber\\ &\qquad\times
	\Lambda_\balpha(a+bp)\bigl(1+J_1(a)\cdot bp+J_2(a)\cdot(bp)^2\bigr)
	\mod{p^3}.
	\label{FSA}
	\end{align}
	
	\subsection{Dirichlet characters}\label{ss:char}
	
	For $r\in \Z_p^\times$, the reflection formula of the $p$-adic Gamma function reads $\G_p(r)\,\G_p(1-\nobreak r)=(-1)^{[r]_0}$.
	In our situation $[r_j]_0=p-[-r_j]_0\equiv a_j\pmod 2$ for $j=1,2,3,4$ and $p>5$,
	therefore
	\begin{equation}
	\chi_{\balpha}(p):=\prod_{j=1}^4\Gamma_p(r_j)
	=\Gamma_p(r_1)\,\Gamma_p(1-r_1)\,\Gamma_p(r_2)\Gamma_p(1-r_2)
	=(-1)^{a_1+a_2}.
	\label{refl}
	\end{equation}
	We will use the result in our further derivations.
	
	\begin{Lemma}\label{lem:char}
		For each multi-set $\balpha=\{r_1,1-r_1,r_2,1-r_2\}$ as in Theorem~\textup{\ref{thm:main}},
		$\chi_{\balpha}$ defines the quadratic Dirichlet character as in the last column of Table~\textup{\ref{tab1}}.
	\end{Lemma}
	
	\begin{proof}
		It is  straightforward that $\G_p(r)\,\G_p(1-r)$ agrees with the Legendre symbol
		$\bigl(\frac{-1}p\bigr)$, $\bigl(\frac{-3}p\bigr)$, $\bigl(\frac{-2}p\bigr)$ and $\bigl(\frac{-1}p\bigr)$ when $r=\frac12,\frac13,\frac14,\frac16$, respectively.
		Furthermore, a direct verification implies that the right-hand side of \eqref{refl} is~1 when $(r_1,r_2)=(\frac1{10},\frac3{10})$ or $(\frac1{12},\frac5{12})$;
		$\bigl(\frac5p\bigr)$ when $(r_1,r_2)=(\frac15,\frac25)$; and
		$\bigl(\frac2p\bigr)$ when $(r_1,r_2)=(\frac18,\frac38)$.
	\end{proof}

	\subsection{Proof of Theorem \ref{thm:3}}\label{ss:4.2}
	
	In view of \eqref{FSA}, Theorem \ref{thm:3} is implied by the following lemma.
	
	\begin{Lemma}
		\label{modp^3-reduction}
		For any $b\in\Z_{\ge0}$, the congruence
		\begin{equation*}
		\sum_{a=0}^{p-1}\frac{\prod_{j=1}^4(r_j)_a}{a!^4}
		\bigl( \Lambda(a+bp)\bigl(1+J_1(a)\cdot bp+J_2(a)\cdot (bp)^2\bigr)-1 \bigr) \equiv 0\mod{p^3}
		\end{equation*}
		holds.
	\end{Lemma}
	
	\begin{proof}
		From the $p$-adic evaluation of \eqref{prodr} and the definition of $\Lambda(a+bp)$, we conclude that
		the left-hand side modulo $p^3$ is a quadratic polynomial $C_0+C_1b+C_2b^2$ in $b$, with the constant term $C_0=0$, and
		\begin{align*}
		C_1&=p\sum_{a=0}^{a_1}\frac{\prod_{j=1}^4(r_j)_a}{a!^4}J_1(a)
		+\sum_{a=a_1+1}^{a_2}\frac{\prod_{j=1}^4(r_j)_a}{a!^4}
		\biggl(\frac1{r_1'}+pJ_1(a)\biggl)
		\\ &\qquad +\biggl(\frac1{r_1'}+\frac1{r_2'}\biggr)\sum_{a=a_2+1}^{a_3}\frac{\prod_{j=1}^4(r_j)_a}{a!^4};
		\displaybreak[2]\\
		C_2&=p^2\sum_{a=0}^{a_1}\frac{\prod_{j=1}^4(r_j)_a}{a!^4}J_2(a)
		+\frac{p}{r_1'}\sum_{a=a_1+1}^{a_2}\frac{\prod_{j=1}^4(r_j)_a}{a!^4}J_1(a)
		\\ &\qquad
		+\frac1{r_1'r_2'}\sum_{a=a_2+1}^{a_3}\frac{\prod_{j=1}^4(r_j)_a}{a!^4},
		\end{align*}
		where the terms, which are zero modulo $p^3$ for trivial reasons, are discarded.
		Our goal is to demonstrate that $C_1\equiv0\mod{p^3}$ and $C_2\equiv0\mod{p^3}$.
		
		Introduce the rational function
		$$
		R(t)=\frac{\prod_{j=1}^4\prod_{i=1}^{a_j}(t-i+pr_j')}{\prod_{i=0}^{p-1}(t+i)^2}.
		$$
		This function is a generalization of the one in \eqref{eq:ratfun}, with the correction terms $pr_j'$ added to make our argument below more efficient.
		The degree of its numerator $a_1+a_2+a_3+a_4=(a_1+a_4)+(a_2+a_3)=(p-1)+(p-1)=2(p-1)$ is by 2 less than the degree $2p$ of its denominator;
		hence it can be represented as the sum of partial fractions,
		$$
		R(t)=\sum_{k=0}^{p-1}\biggl(\frac{A_k}{(t+k)^2}+\frac{B_k}{t+k}\biggr),
		$$
		for which the identity
		\begin{equation}
		\label{sum-res}
		\sum_{k=0}^{p-1}B_k=\sum_{k=0}^{p-1}\Res_{t=-k}R(t)=-\Res_{t=\infty}R(t)=0
		\end{equation}
		is implied by the residue sum theorem.
		The coefficients in the partial-fraction decomposition can be computed explicitly:
		\begin{align*}
		A_k&=R(t)(t+k)^2\big|_{t=-k}
		=\frac{\prod_{j=1}^4\prod_{i=1}^{a_j}(k+i-pr_j')}{k!^2(p-1-k)!^2}
		=\frac{\prod_{j=1}^4(k+1-pr_j')_{a_j}}{(1)_k^2(1)_{p-1-k}^2},
		\\
		B_k&=A_k\biggl(-\sum_{j=1}^4\sum_{i=1}^{a_j}\frac1{k+i-pr_j'}+2H_k-2H_{p-1-k}\biggr),
		\end{align*}
		where as before $H_k$ stands for the $k$-th partial harmonic sum.
		Using equation~\eqref{(t)_k} and the similar transformation
		\begin{align*}
		(k+1-p\lambda)_a
		&=(-1)^a\frac{\Gamma_p(k+1+a-p\lambda)}{\Gamma_p(k+1-p\lambda)}\,(k+1-p\lambda+[-(k+1-p\lambda)]_0)^{\nu(a,[-k-1]_0)}
		\\
		&=(-1)^a\frac{\Gamma_p(k+1+a-p\lambda)}{\Gamma_p(k+1-p\lambda)}\,(p(1-\lambda))^{\nu(a,[-k-1]_0)}
		\end{align*}
		for $0\le k\le p-1$ and $0< a<p$, as well as noticing that
		$$
		\nu(a_j,[-k-1]_0)=\nu(k+1,[-a_j]_0)=\nu(k+1,[r_j]_0)
		$$
		for $j=1,2,3,4$, we find out that
		\begin{align*}
		A_k&=\frac{\prod_{j=1}^4\Gamma_p(k+1+a_j-pr_j')\,(p(1-r_j'))^{\nu(a_j,[-k-1]_0)}}{\Gamma_p(k+1)^2\Gamma_p(p-k)^2\prod_{j=1}^4\Gamma_p(k+1-pr_j')}
		\\
		&=\frac{\Gamma_p(k+1-p)^2\prod_{j=1}^4\Gamma_p(k+1-r_j)\,(p(1-r_j'))^{\nu(k+1,[r_j]_0)}}{\Gamma_p(k+1)^2\prod_{j=1}^4\Gamma_p(k+1-pr_j')}
		\displaybreak[2]\\
		&=\frac{\prod_{j=1}^4(-r_j)_{k+1}}{k!^4}\,\frac{\Gamma_p(k+1)^2\Gamma_p(k+1-p)^2\prod_{j=1}^4\Gamma_p(-r_j)}{\prod_{j=1}^4\Gamma_p(k+1-pr_j')}
		\displaybreak[2]\\
		&=\frac{r_1r_2r_3r_4\prod_{j=1}^4(1-r_j)_k}{k!^4}\,\frac{\Gamma_p(k+1)^2\Gamma_p(k+1-p)^2\prod_{j=1}^4\Gamma_p(1-r_j)}{r_1r_2r_3r_4\prod_{j=1}^4\Gamma_p(k+1-pr_j')}
		\displaybreak[2]\\
		&=\frac{\prod_{j=1}^4(r_j)_k}{k!^4}\,\prod_{j=1}^4\Gamma_p(r_j)\,(1+O(p^2))
		\\
		&=(-1)^{a_1+a_2}\frac{\prod_{j=1}^4(r_j)_k}{k!^4}\,(1+O(p^2)).
		\end{align*}
		Furthermore,
		\begin{align*}
		&
		-\sum_{j=1}^4\sum_{i=1}^{a_j}\frac1{k+i-pr_j'}+2H_k-2H_{p-1-k}
		\\ &\quad
		=-\sum_{j=1}^4\biggl(G_1(k+1+a_j-pr_j')-G_1(k+1-pr_j')+\frac{\nu(a_j,[-k-1]_0)}{p(1-r_j')}\biggr)
		+2H_k-2H_{p-1-k}
		\displaybreak[2]\\ &\quad
		=-\biggl(\sum_{j=1}^4G_1(k+1-r_j)-4G_1(k+1)\biggr)
		-\sum_{j=1}^4\frac{\nu(k+1,[-a_j]_0)}{p(1-r_j')}
		\\ &\quad\qquad
		+\sum_{j=1}^4G_1(k+1-pr_j')-4G_1(1)-2H_k-2H_{p-1-k}
		\displaybreak[2]\\ &\quad
		\equiv-\biggl(\sum_{j=1}^4G_1(k+1-r_j)-4G_1(k+1)\biggr)
		-\sum_{j=1}^4\frac{\nu(k+1,[-a_j]_0)}{p(1-r_j')}\mod{p^2}
		\\ &\quad
		=-\biggl(J_1(k)
		+\sum_{j=1}^4\frac{\nu(k+1,[-a_j]_0)}{p(1-r_j')}\biggr),
		\end{align*}
		so that
		$$
		B_k\equiv(-1)^{a_1+a_2+1}\frac{\prod_{j=1}^4(r_j)_k}{k!^4}J_1(k)\mod{p^2}
		$$
		for $0\le k\le a_1$;
		$$
		B_k\equiv(-1)^{a_1+a_2+1}\frac{\prod_{j=1}^4(r_j)_k}{k!^4}\biggl(J_1(k)+\frac1{p(1-r_4')}\biggr)\mod{p^2}
		$$
		for $a_1<k\le a_2$;
		\begin{align*}
		B_k&\equiv(-1)^{a_1+a_2+1}\frac{\prod_{j=1}^4(r_j)_k}{k!^4}\biggl(J_1(k)
		+\frac1{p(1-r_4')}+\frac1{p(1-r_3')}\biggr)\mod{p^2}
		\end{align*}
		for $a_2<k\le a_3$; and $B_k\equiv0\mod{p^2}$ for $k>a_3$. Since $1-r_4'=r_1'$ and $1-r_3'=r_2'$, we obtain
		\begin{align*}
		(-1)^{a_1+a_2+1}\sum_{k=0}^{p-1}B_k
		&\equiv\sum_{k=0}^{a_3}\frac{\prod_{j=1}^4(r_j)_k}{k!^4}J_1(k)
		\\ &\qquad
		+\frac1{pr_1'}\sum_{k=a_1+1}^{a_3}\frac{\prod_{j=1}^4(r_j)_k}{k!^4}
		+\frac1{pr_2'}\sum_{k=a_2+1}^{a_3}\frac{\prod_{j=1}^4(r_j)_k}{k!^4}
		\mod{p^2}.
		\end{align*}
		By comparing this with \eqref{sum-res} and the formula defining $C_1$ we conclude that $C_1\equiv0\mod{p^3}$.
		
		We next show that $C_2\equiv 0\mod p^3$ using the different rational function
		$$
		\wt R(t)
		=\frac{\prod_{i=1}^{a_1}(t-i)}{\prod_{j=1}^3\prod_{i=0}^{a_j}(t+i)}
		=\sum_{k=0}^{a_1}\frac{\wt A_k}{(t+k)^3}+\sum_{k=0}^{a_2}\frac{\wt B_k}{(t+k)^2}+\sum_{k=0}^{a_3}\frac{\wt D_k}{t+k}
		$$
		and the related residue-sum identity
		\begin{equation}
		\label{sum-res1}
		\sum_{k=0}^{a_3}\wt D_k=\sum_{k=0}^{a_3}\Res_{t=-k}\wt R(t)=-\Res_{t=\infty}\wt R(t)=0
		\end{equation}
		for it. By construction, $\wt R(t)$ only has poles of order $j\in\{1,2,3\}$ at the points $t=-k$ with $a_{j-1}<k\le a_j$ (recall the additional setting $a_0=-1$).
		With the argument used in the proof of Lemma \ref{gp-poch} and reflection formula for the $p$-adic Gamma function we record
		\begin{align*}
		\prod_{i=1}^{a_1}(t-i)\bigg|_{t=-k}
		&=(-1)^{k+1}\frac{\Gamma_p(k+1+a_1)}{k!}=\frac{\Gamma_p(k+1-r_1)}{\Gamma_p(k+1)}\,(1+O(p))
		\\
		&=\phantom-\frac{\Gamma_p(k+r_4)}{\Gamma_p(k+1)}\,(1+O(p))
		\quad\text{for}\; 0\le k\le a_3
		\\ \intertext{and}
		\frac{t+k}{\prod_{i=0}^{a_j}(t+i)}\bigg|_{t=-k}
		&=\frac{(-1)^k}{k!\,(a_j-k)!}=\frac{(-1)^{a_j-k}}{\Gamma_p(k+1)\,\Gamma_p(a_j-k+1)}=-\frac{\Gamma_p(k-a_j)}{\Gamma_p(k+1)}
		\\
		&=-\frac{\Gamma_p(k+r_j)}{\Gamma_p(k+1)}\,(1+O(p))
		\quad\text{for}\; 0\le k\le a_j;
		\displaybreak[2]\\
		\frac1{\prod_{i=0}^{a_j}(t+i)}\bigg|_{t=-k}
		&=\frac1{(-k)_{a_j+1}}=\frac{(-1)^{a_j+1}\Gamma_p(-k)}{\Gamma_p(a_j-k+1)}=\frac{\Gamma_p(k-a_j)}{\Gamma_p(k+1)}
		\\
		&=\phantom-\frac{\Gamma_p(k+r_j)}{\Gamma_p(k+1)}\,(1+O(p))
		\quad\text{for}\; a_j<k\le a_3,
		\end{align*}
		where $j=1,2,3$. Similar formulas but involving the functions $G_1$ and $G_2$ are valid for the $t$-derivatives of the left-hand sides and afterwards
		substitution $t=-k$, because all the terms in these formulas belong to $\mathbb Z_p^\times$: for any $s=0,1,2,\dots$,
		\begin{align*}
		\frac1{s!}\,\frac{\d^s}{\d t^s}\biggl(\prod_{i=1}^{a_1}(t-i)\biggr)\bigg|_{t=-k}
		&\equiv\phantom-(-1)^s\frac1{s!}\,\frac{\d^s}{\d t^s}\biggl(\frac{\Gamma_p(t+r_4)}{\Gamma_p(t+1)}\biggr)\bigg|_{t=k}\mod p
		\quad\text{for}\; 0\le k\le a_3;
		\displaybreak[2]\\
		\frac1{s!}\,\frac{\d^s}{\d t^s}\biggl(\frac{t+k}{\prod_{i=0}^{a_j}(t+i)}\biggr)\bigg|_{t=-k}
		&\equiv-(-1)^s\frac1{s!}\,\frac{\d^s}{\d t^s}\biggl(\frac{\Gamma_p(t+r_j)}{\Gamma_p(t+1)}\biggr)\bigg|_{t=k}\mod p
		\quad\text{for}\; 0\le k\le a_j;
		\displaybreak[2]\\
		\frac1{s!}\,\frac{\d^s}{\d t^s}\biggl(\frac1{\prod_{i=0}^{a_j}(t+i)}\biggr)\bigg|_{t=-k}
		&\equiv\phantom-(-1)^s\frac1{s!}\,\frac{\d^s}{\d t^s}\biggl(\frac{\Gamma_p(t+r_j)}{\Gamma_p(t+1)}\biggr)\bigg|_{t=k}\mod p
		\quad\text{for}\; a_j<k\le a_3.
		\end{align*}
			This computation implies that
		\begin{align*}
		\wt D_k
		&=\frac1{(j-1)!}\,\frac{\d^{j-1}}{\d t^{j-1}}\bigl(\wt R(t)(t+k)^j\bigr)\bigg|_{t=-k}
		\\
		&\equiv-\frac1{(j-1)!}\,\frac{\d^{j-1}}{\d t^{j-1}}\biggl(\frac{\prod_{\ell=1}^4\Gamma_p(r_\ell+t)}{\Gamma_p(1+t)^4}\biggr)\bigg|_{t=k}
		\mod p
		\end{align*}
		for $j=1,2,3$ and $a_{3-j}<k\le a_{4-j}$; therefore,
		\begin{align*}
		-\sum_{k=0}^{a_3}\wt D_k
		&=\sum_{k=0}^{a_1}\frac12\,\frac{\d^2}{\d t^2}\biggl(\frac{\prod_{\ell=1}^4\Gamma_p(r_\ell+t)}{\Gamma_p(1+t)^4}\biggr)\bigg|_{t=k}
		\\ &\qquad
		+\sum_{k=a_1+1}^{a_2}\frac{\d}{\d t}\biggl(\frac{\prod_{\ell=1}^4\Gamma_p(r_\ell+t)}{\Gamma_p(1+t)^4}\biggr)\bigg|_{t=k}
		\\ &\qquad
		+\sum_{k=a_2+1}^{a_3}\frac{\prod_{\ell=1}^4\Gamma_p(r_\ell+t)}{\Gamma_p(1+t)^4}\bigg|_{t=k}
		\mod p.
		\end{align*}
		Expanding the derivatives on the right-hand side of the resulting equality { using both identities of Lemma \ref{gp-poch}} and \eqref{eq:GammaRis} we arrive at
		\begin{align*}
		\frac{\prod_{\ell=1}^4\Gamma_p(r_\ell+t)}{\Gamma_p(1+t)^4}\bigg|_{t=k}
		&=\frac{\prod_{\ell=1}^4\Gamma_p(r_\ell)}{k!^4}\,\frac{\prod_{\ell=1}^4\Gamma_p(r_\ell+k)}{\prod_{\ell=1}^4\Gamma_p(r_\ell)}
		\\
		&=\frac{(-1)^{a_1+a_2}}{k!^4}\biggl(\frac1{pr_1'}\biggr)^{\nu(k,a_1)}\biggl(\frac1{pr_2'}\biggr)^{\nu(k,a_2)}\prod_{\ell=1}^4(r_\ell)_k.
		\end{align*}
		It follows that
		$$
		-p^2\sum_{k=0}^{a_3}\wt D_k\equiv(-1)^{a_1+a_2}C_2\mod{p^3}.
		$$
		Thus, the residue sum formula \eqref{sum-res1} implies the desired congruence
		$C_2\equiv0\mod{p^3}$ and completes the proof of Lemma~\ref{modp^3-reduction}.
	\end{proof}
	
	\subsection{Companion congruences}\label{ss:4.3}
	We now record two companion congruences which will be
	employed in our second proof of Theorem \ref{thm:main},
	for all primes $p>5$ (not necessary ordinary).
	This is a bi-product of the derivation above, when the residue sum computation is performed for the rational functions $tR(t)$ and $t\wt R(t)$
	in place of $R(t)$ and $\wt R(t)$,
	on using $\Res_{t=\infty}tR(t)=1$ (respectively, $\Res_{t=\infty}t\wt R(t)=0$)
	as well as \eqref{refl}. In these settings, the analysis in Section~\ref{ss:4.2} reveals us with the following claim.
	
	\begin{Lemma}\label{cor:1}
		We have
		\begin{align*}
		\wt C_1&:=p\sum_{k=0}^{a_2}\frac{\prod_{j=1}^4(r_j)_k}{k!^4}\,(J_1(k)k+1)
		\\&\qquad
		+\frac{1}{r_1'}\sum_{k=a_1+1}^{a_3}\frac{\prod_{j=1}^4(r_j)_k}{k!^4}k
		+\frac{1}{r_2'}\sum_{k=a_2+1}^{a_3}\frac{\prod_{j=1}^4(r_j)_k}{k!^4}k
		\\
		&\phantom:\equiv (-1)^{a_1+a_2}p \mod p^3;
		\displaybreak[2]\\
		\wt C_2&:=p^2\sum_{k=0}^{a_1}\frac{\prod_{j=1}^4(r_j)_k}{k!^4}\(J_2(k)k^2+J_1(k)k\)
		\\ &\qquad
		+p\sum_{k=a_1+1}^{a_2}\frac{\prod_{j=1}^4(r_j)_k}{k!^4}
		(J_1(k)k+1)\frac{k}{r_1'}
		+\frac1{r_1'r_2'}\sum_{k=a_2+1}^{a_3}\frac{\prod_{j=1}^4(r_j)_k}{k!^4}k^2
		\\
		&\phantom: \equiv-(-1)^{a_1+a_2}p^2 \mod p^3.
		\end{align*}
	\end{Lemma}
	
	We will use these congruences later in Section~\ref{ss:charsum} for our different treatment of Theorem~\ref{thm:main},
	where $\wt C_1$ and $\wt C_2$ will play a role similar to that of $C_1$ and $C_2$ in the proof of Lemma~\ref{modp^3-reduction}.
	
	\section{Hypergeometric motives and the modularity of~rigid~Calabi--Yau~threefolds}
	\label{ss:charactersums}
	
	It is now timely to formally introduce hypergeometric motives and their connection to both classical and finite hypergeometric functions, and to $L$-functions.
	The concept was proposed by Katz in \cite{Katz} and developed into a fundamental object of study by Roberts, Rodriguez-Villegas and Watkins \cite{RVRW}.
	The main emphasis of such developments is on hypergeometric motives that admit lifts to~$\Q$; all cases considered below fall into this category.
	
	In the formulation of hypergeometric motives defined over $\Q$, the data consist of
	a parameter $\l\in \Q^\times$ and two multi-sets
	$$
	\balpha=\{\alpha_1,\dots, \alpha_m\}\quad\text{and}\quad
	\bbeta=\{\beta_1,\dots, \beta_m\}
	$$
	from $\Q^m$, each closed under the Galois conjugation, that is,
	$\prod_{j=1}^m(x-e^{2\pi i\alpha_j})\in \Z[x]$ and $\prod_{j=1}^m(x-e^{2\pi i\beta_j})\in \Z[x]$.
	Recall that the closedness of $\balpha$ under the Galois conjugation implies that
	$c\balpha\equiv\balpha\mod{\Z^m}$ as multi-sets, for any non-zero $c\in\Z$ relatively prime to the common denominator of $\alpha_1,\dots,\alpha_m$.
	
	Additionally, assume that the hypergeometric data are \emph{primitive} meaning that $\alpha_j-\beta_\ell\notin \Z$ for all $j$ and $\ell$.
	Under the latter constraint, there is an explicit way to numerically compute the $L$-function of generic degree $m$ attached to the hypergeometric data;
	the details are recorded in the notes \cite{RVRW} and the related \texttt{Magma} documentation on hypergeometric motives \cite{Watkins} by Watkins.
	
	In the classical setting, when $\beta_m=1$ and none of $\beta_j$ is a non-negative integer,
	the corresponding (generalized) hypergeometric function is defined as
	\begin{align}
	\PFQ{m}{m-1}{\alpha_1,\,\alpha_2,\,\dots,\,\alpha_m}{\beta_1,\,\dots,\,\beta_{m-1}}{\l}
	&=\sum_{k=0}^\infty\prod_{j=1}^m\frac{(\alpha_j)_k}{(\beta_j)_k}\,\l^k
	\nonumber\\
	&=\sum_{k\in\mathbb Z}\prod_{j=1}^m\frac{\G(\alpha_j+k)}{\G(\alpha_j)}\,\frac{\G(1-\beta_j-k)}{\G(1-\beta_j)}\bigl((-1)^m\l\bigr)^k,
	\label{eq:HGS}
	\end{align}
	where $(\alpha)_k=\G(\alpha+k)/\G(\alpha)$ as in the introduction.
	
	We stress on the fact that the above conditions on hypergeometric data are invariant under shifts of any $\alpha_j$ or $\beta_j$ by integers;
	the invariance will be also featured by the finite-field analogue of the hypergeometric function. However, the hypergeometric function \eqref{eq:HGS} itself
	is sensitive to such integer shifts\,---\,this is the subject of contiguous relations in the classical setting \cite[Section 2.5]{AAR}. To avoid unnecessary
	sophistication we will always assume that the multi-sets $\balpha$ and $\bbeta$ are normalized so that $0<\alpha_j\le1$ and $0<\beta_j\le1$ for $j=1,\dots,m$.
	This is the case \eqref{hyp:main} we deal with in our applications.
	
	\subsection{Finite hypergeometric functions}\label{ss:3.1}
	The theory of finite hypergeometric functions was initiated by Greene \cite{Greene} and Katz \cite{Katz},
	with considerable developments over the recent years\,---\,see \cite{BCM, McCarthy-Pacific, RV2}.
	An interpretation of the theory in connection with Galois representations \cite{Katz, WIN3X}
	yields some fruitful results in computing zeta functions of algebraic varieties defined over finite fields.
	
	We first fix some more notation. Let $\F_q$ denote the finite field of size $q$ and $\theta$ a fixed non-trivial additive character of $\F_q$.
	Let $\widehat{\F_q^\times}$ be the set of all multiplicative characters of $\F_q$ with values in $\mathbb C_p^\times$, which is a cyclic group of size $q-1$,
	and $\omega$ a generator of $\widehat {\F_q^\times}$, the Teichm\"uller character, so that $\widehat{\F_q^\times}=\{\omega^k\}_{k=0}^{q-2}$.
	Denote by $\eps$ the trivial multiplicative character.
	For any character $\chi\in\widehat{\F_q^\times}$ including $\eps$, we adapt the convention that $\chi(0)=0$.
	Define the Gauss sum of $\chi$ by
	$$
	g(\chi):=\sum_{x\in\F_q}\chi(x)\theta(x).
	$$

	Under the hypothesis $(q-1)\alpha_j,(q-1)\beta_j\in \Z$ for $j=1,\dots,m$,  Katz defines in \cite[Section~8.2]{Katz} the  character sum
	\begin{equation}\label{eq:Katz}
	H_q(\balpha,\bbeta;\l)^K:=\frac1{q-1}\sum_{k=0}^{q-2}\omega^{-k}(\l)
	\prod_{j=1}^m g(\omega^{k+(q-1)\alpha_j})g(\omega^{-k-(q-1)\beta_j})\,\omega^{k+(q-1)\beta_j}(-1).
	\end{equation}
	Following \cite[Definition 1.1]{BCM}, we use the following modification.
	It corresponds to Katz's   ``canonical'' version (independent of the choice of an additive character) in his later paper \cite[Section~4]{Katz09}: 
	\begin{equation}\label{eq:H}
	H_q(\balpha,\bbeta;\l):=\frac1{1-q}\sum_{k=0}^{q-2}
	\prod_{j=1}^m\left (\frac{g(\omega^{k+(q-1)\alpha_j})g(\omega^{-k-(q-1)\beta_j})}{g(\omega^{(q-1)\alpha_j})g(\omega^{-(q-1)\beta_j})}\right)\,\omega^k\bigl((-1)^m\l\bigr).
	\end{equation}
	For $\balpha=\{r_1,1-r_1,r_2,1-r_2\}$, $\beta=\{1,1,1,1\}$ and $\l=1$ as in our context, we have
	\begin{equation}\label{eq:weight}H_q(\balpha,\beta;1)=-\omega^{(q-1)(r_1+r_2)}(-1)\cdot H_q(\balpha,\beta;1)^K/q^2.
	\end{equation}
	
	Gauss sums are known to be finite-field analogues of the Gamma function values (see~\cite{WIN3X} for a dictionary between the classical and finite-field settings),
	therefore, the right-hand side of \eqref{eq:H} reminisces \eqref{eq:HGS}.  Next, we follow \cite{BCM} to extend the definition to almost all finite fields $\F_q$,
	without the restriction on~$q$ to satisfy $(q-1)\alpha_j\in\Z$ and $(q-1)\beta_j\in \Z$ for $j=1,\dots,m$.

	\subsection{Finite hypergeometric sums}
	
	For a given integer $d\ge1$, let
	$$
	\varphi(d)=\sum_{n\mid d}\mu\biggl(\frac dn\biggr)n
	$$
	denote the Euler number of $d$, where $\mu(\,\cdot\,)$ is the M\"obius function. Write
	$$
	\frac 1{(x-1)^{\varphi(d)}}\cdot\prod_{n\mid d}(x^n-1)^{\mu(d/n)}=\frac{\prod_{i=1}^r(x^{p_i}-1)}{\prod_{j=1}^s (x^{q_j}-1)},
	$$
	where $\{p_i\}$, $\{q_j\}$ are disjoint multi-sets of positive integers, and define
	\begin{equation}\label{eq:Md}
	M_d:=\prod_{n\mid d}(n^n)^{\mu(d/n)}.
	\end{equation}
	Then it is not hard to verify (see \cite{BCM}) that for $k\in \mathbb Z_{>0}$,
	\begin{equation}\label{eq:24}
	\prod_{\substack{\ell=1\\(\ell,d)=1}}^d\frac{(\ell/d)_k}{k!}
	=\frac{(p_1k)!\cdots (p_rk)!}{(q_1k)!\cdots (q_sk)!}\,M_d^{-k}
	=\frac{1}{k!^{\varphi(d)}} \prod_{n\mid d}\bigl((nk)!\,n^{-nk}\bigr)^{\mu(d/n)}
	\end{equation}
	
	Now define a function $S_d(\chi)$ on characters $\chi\in \widehat{\F_q^\times}$ by
	\begin{equation}\label{eq:gauss}
	S_d(\chi):= g(\chi^{-1})^{\varphi(d)}\cdot\prod_{n\mid d}\bigl(g(\chi^n)\,\chi(n^{-n})\bigr)^{\mu(d/n)}.
	\end{equation}
	Noting that $k!=\Gamma(k+1)$ and $(nk)!=\Gamma(nk+1)$ we see that this is a finite-field analogue of the last term in \eqref{eq:24}.
	
	Notice that each multi-set $\balpha=\{\alpha_1,\dots,\alpha_m\}$ closed under the Galois conjugation and satisfying $0<\alpha_j\le1$ for $j=1,\dots,m$,
	can be partitioned in the form
	\begin{equation}\label{eq:partition}
	\balpha=\bigcup_{i=1}^t\Sigma_{d_i},
	\quad\text{where}\; \Sigma_d=\{\ell/d:0<\ell\le d,\ (\ell,d)=1\} \;\text{for}\; d\in\{d_1,\dots,d_t\},
	\end{equation}
	with $\varphi(d_1)+\dots+\varphi(d_t)=m$. The examples $\balpha=\{r_1,1-r_1,r_2,1-r_2\}$ showing up in Theorem~\ref{thm:main}
	all come from the multi-sets $\Sigma_d$ with $\varphi(d)=1$, $2$ or~$4$; the related data are collected in Table~\ref{tab2}.
	Our next proposition is a version of a special case of Theorem~1.3 in \cite{BCM}, in which we escape the use of reflection formula for the Gauss sums,
	so that denominators are kept in the intermediate character sums as they are.
	
	\begin{Proposition}\label{prop:2}
		Assume that the hypergeometric data consist of $\l\in\Q^\times$, $\balpha=\bigcup_{i=1}^t\Sigma_{d_i}$ and $\bbeta=\{1,\dots,1\}$, and are primitive.
		Then for any finite field $\F_q$ with $q\equiv 1\mod d_i$ for all~$i$,
		\begin{equation}\label{eq:H3}
		H_q(\balpha, \bbeta;\l)=\frac{1}{1-q}\sum_{k=0}^{q-2}\prod_{i=1}^t S_{d_i}(\omega^k)\,\omega^k\bigl((-1)^m\l\bigr).
		\end{equation}
	\end{Proposition}
	
	\begin{proof}
		First, by \cite{BCM} the finite hypergeometric function in \eqref{eq:H} can be rewritten as
		$$
		H_q(\balpha, \bbeta;\l)=\frac{1}{1-q}\sum_{k=0}^{q-2}\prod_{i=1}^t \mathsf P_{d_i}(k)\,\omega^k\bigl((-1)^m\l\bigr),
		$$
		where
		$$
		\mathsf P_d(k)
		=\prod_{\ell\in(\Z/d\Z)^\times}\frac{g(\omega^{k+(q-1)\ell/d})g(\omega^{-k})}{g(\omega^{(q-1)\ell/d})g(\eps)}
		=g(\omega^{-k})^{\varphi(d)}\prod_{\ell\in(\Z/d\Z)^\times}\frac{g(\omega^{k+(q-1)\ell/d})}{g(\omega^{(q-1)\ell/d})g(\eps)}.
		$$
		Following the proof of \cite[Theorem 1.3]{BCM}, we can reformulate $\mathsf P_d(k)$ as
		$$
		\mathsf P_d(k)=g(\omega^{-k})^{\varphi(d)}\prod_{n\mid d}\bigl(g(\omega^{nk})\omega^k(n^{-n})\bigr)^{\mu(d/n)}
		$$
		using the multiplication formula of Gauss sums (also known as the Hasse--Davenport relations \cite{BEW, BCM}).
		Notice that the function $\mathsf P_d(k)$ agrees with $S_d(\omega^k)$, hence the desired identity follows.
	\end{proof}
	
	We remark that $H_q(\balpha,\bbeta;\l)$ can be also given directly via Gauss sums (see \cite[Definition~1.1]{BCM}) and that the right-hand side of \eqref{eq:H3} is well defined for any finite field $\F_q$.
	Following \cite{BCM}, we use this more general definition of $H_q(\balpha, \bbeta;\l)$ in further discussions.
	
	\begin{table}[h]
		\caption{Cyclotomic data}
		\begin{center}
			$$
			\begin{array}{|c||c|c|c|c|c|}
			\hline
			d&\displaystyle \frac{\prod_{i=1}^r(x^{p_i}-1)\vphantom{|^{0^0}}}{\prod_{j=1}^s (x^{q_j}-1)}& (r,s) & \{p_i\},\ \{q_j\}& M_d \ \text{(see \eqref{eq:Md})} \\ \hline\hline
			2& \displaystyle\frac{x^2-1\vphantom{|^{0^0}}}{(x-1)^2}&  (1,2)& \begin{array}{c} p_1=2 \\ q_1=q_2=1\end{array}&2^2\\\hline
			3& \displaystyle\frac{x^3-1\vphantom{|^{0^0}}}{(x-1)^3}&  (1,3)& \begin{array}{c} p_1=3 \\ q_1=q_2=q_3=1\end{array}&3^3\\\hline
			4& \displaystyle\frac{x^4-1\vphantom{|^{0^0}}}{(x-1)^2(x^2-1)}&  (1,3)& \begin{array}{c} p_1=4 \\ q_1=q_2=1,q_3=2 \end{array}&2^6\\\hline
			6& \displaystyle\frac{x^6-1\vphantom{|^{0^0}}}{(x-1)(x^2-1)(x^3-1)}&  (1,3)& \begin{array}{c} p_1=6 \\ q_1=1,q_2=2,q_3=3\end{array}&2^4\cdot 3^3\\\hline
			5& \displaystyle\frac{x^5-1\vphantom{|^{0^0}}}{(x-1)^5}&  (1,5)& \begin{array}{c} p_1=5 \\ q_1=\dots =q_5=1\end{array}&5^5\\\hline
			8& \displaystyle\frac{x^8-1\vphantom{|^{0^0}}}{(x-1)^4(x^4-1)}&  (1,5)& \begin{array}{c} p_1=8 \\ q_1=\dots=q_4=1,\\ q_5=4\end{array}&2^{16}\\ \hline
			10& \displaystyle\frac{x^{10}-1}{\(x-1\)^3(x^2-1)(x^5-1)}& (1,5)& \begin{array}{c} p_1=10 \\ q_1=q_2=q_3=1,\\ q_4=2, q_5=5\end{array}&2^{8}\cdot 5^5\\ \hline
			12&\displaystyle\frac{\(x^{12}-1\)(x^2-1)}{\(x-1\)^4(x^4-1)(x^6-1)}& (2,6)& \begin{array}{c} p_1=2, p_2=12 \\ q_1=\dots=q_4=1,\\ q_5=4, q_6=6\end{array}&2^{12}\cdot 3^6\\
			\hline
			\end{array}
			$$
		\end{center}
		\label{tab2}
	\end{table}
	
	\subsection{Reduction to $p$-adic Gamma sums}\label{ss:5.1}
	The Gross--Koblitz formula \cite{GK} expresses Gauss sums $g(\chi)$ in terms of the $p$-adic Gamma function: for $0\le k\le p-2$,
	$$
	g(\omega^{-k})=-\pi_p^k\G_p\biggl(\frac{k}{p-1}\biggr),
	$$
	where $\pi_p$ is a fixed root of $x^{p-1}+p=0$ in $\mathbb C_p$.
	For $p>2$ and an integer $n$ relatively prime with $p$,
	the product formula of the $p$-adic Gamma function \cite{Cohen, McCarthy-RV} reads
	\begin{equation}
	\prod_{\ell=0}^{n-1}\Gamma_p\biggl(\frac{x+\ell}n\biggr)
	=\omega(n^{(1-x)(1-p)})\,\Gamma_p(x)\prod_{\ell=1}^{n-1}\Gamma_p\biggl(\frac\ell n\biggr),
	\label{prodG}
	\end{equation}
	where $x$ is of the form $m/(p-1)$ with $0\le m<p-1$.
	
	\begin{Lemma}
		\label{lem:aux}
		Given a prime $p>2$ and an integer $d>1$ such that $(d,p)=1$,
		the function $\nu(k,x)$ defined in \eqref{eq:nu} satisfies
		$$
		\sum_{\substack{\ell=1\\(\ell,d)=1}}^d\nu\biggl(k,\frac\ell d(p-1)\biggr)
		=-\sum_{n\mid d}\mu\biggl(\frac dn\biggr)\biggl\lfloor\frac{-nk}{p-1}\biggr\rfloor.
		$$
	\end{Lemma}
	
	\begin{proof}
		Consider
		$$
		f(x)=\nu(k,x(p-1))=-\biggl\lfloor x-\frac k{p-1}\biggr\rfloor
		$$
		and notice that
		$$
		\sum_{\ell=0}^{n-1}f\biggl(\frac\ell n\biggr)
		=-\biggl\lfloor-\frac{nk}{p-1}\biggr\rfloor,
		$$
		since $\lfloor x\rfloor+\bigl\lfloor x+\frac1n\bigr\rfloor+\dots+\bigl\lfloor x+\frac{n-1}n\bigr\rfloor=\lfloor nx\rfloor$
		(see \cite[Division~8, Problem~9]{PS}). It remains to combine this summation with the combinatorial identity
		\begin{equation}
		\sum_{n\mid d}\mu\biggl(\frac dn\biggr)\sum_{\ell=0}^{n-1}f\biggl(\frac\ell n\biggr)
		=\sum_{n\mid d}\mu\biggl(\frac dn\biggr)\sum_{\ell=1}^{n-1}f\biggl(\frac\ell n\biggr)
		=\sum_{\substack{\ell=1\\(\ell,d)=1}}^df\biggl(\frac\ell d\biggr)
		\label{eq:comb}
		\end{equation}
		valid for a generic function $f(x)$ defined on the interval $0\le x<1$ (see \cite[Division~8, Problem~35]{PS}).
	\end{proof}
	
	\begin{Lemma}
		\label{lem:Sd}
		Let an integer $d\ge2$ be relatively prime with $p>2$.
		For the character sum defined in \eqref{eq:gauss} and an integer $k$, $0\le k\le p-2$,
		$$
		S_d(\omega^{k})=(-1)^{k\varphi(d)}
		\prod_{\substack{\ell=1\\(\ell,d)=1}}^d\frac{\Gamma_p\bigl(\bigl\{\frac\ell d-\frac k{p-1}\bigr\}\bigr)}
		{\Gamma_p\bigl(\frac\ell d\bigr)\,\Gamma_p\bigl(1-\frac{k}{p-1}\bigr)}(-p)^{\nu(k,\ell(p-1)/d)},
		$$
		where $\nu(k,x)$ is as in \eqref{eq:nu}.
	\end{Lemma}
	
	\begin{proof}
		Substitution of the Gross--Koblitz formula into \eqref{eq:gauss} implies
		$$
		S_d(\omega^k)
		=(-1)^{\varphi(d)}\pi_p^{k\varphi(d)}\Gamma_p\biggl(\frac k{p-1}\biggr)^{\varphi(d)}
		\prod_{n\mid d}\biggl(-\pi_p^{(p-1)\{\frac{-nk}{p-1}\}}\Gamma_p\biggl(\biggl\{\frac{-nk}{p-1}\biggr\}\biggr)\omega^k(n^{-n})\biggr)^{\mu(d/n)}.
		$$
		First notice that the exponent of $\pi_p$ on the right-hand side is equal to
		\begin{align*}
		&
		k\varphi(d)+(p-1)\sum_{n\mid d}\mu\biggl(\frac dn\biggr)\biggl\{\frac{-nk}{p-1}\biggr\}
		\\ &\qquad
		=k\varphi(d)+(p-1)\sum_{n\mid d}\mu\biggl(\frac dn\biggr)\cdot\frac{(-nk)}{p-1}
		-(p-1)\sum_{n\mid d}\mu\biggl(\frac dn\biggr)\biggl\lfloor\frac{-nk}{p-1}\biggr\rfloor
		=(p-1)\nu_d(k),
		\end{align*}
		where by Lemma~\ref{lem:aux}
		$$
		\nu_d(k)
		:=\sum_{\substack{\ell=1\\(\ell,d)=1}}^d\nu\biggl(k,\frac\ell d(p-1)\biggr),
		$$
		reducing the character sum to
		\begin{equation}
		S_d(\omega^k)
		=(-1)^{\varphi(d)}(-p)^{\nu_d(k)}\Gamma_p\biggl(\frac k{p-1}\biggr)^{\varphi(d)}
		\prod_{n\mid d}\biggl(\Gamma_p\biggl(\biggl\{\frac{-nk}{p-1}\biggr\}\biggr)\omega^k(n^{-n})\biggr)^{\mu(d/n)}.
		\label{Sd1}
		\end{equation}
		
		For a real number $y$,  the two sets
		$$
		\bigcup_{\ell=0}^{n-1}\frac{\{ny\}+\ell}n
		=\bigcup_{\ell=0}^{n-1}\biggl(y+\frac{\ell-\lfloor ny\rfloor}n\biggr)
		\quad\text{and}\quad
		\bigcup_{\ell=0}^{n-1}\biggl(y+\frac\ell n\biggr)
		$$
		coincide modulo $\mathbb Z$. Because the first set involves $n$ numbers equally distributed in the $n$ equal subintervals of $0\le x<1$, we conclude that
		$$
		\bigcup_{\ell=0}^{n-1}\frac{\{ny\}+\ell}n
		=\bigcup_{\ell=0}^{n-1}\biggl\{y+\frac\ell n\biggr\}.
		$$
		Therefore, application of the product formula \eqref{prodG} with $x=\bigl\{\frac{-nk}{p-1}\bigr\}$ results in
		\begin{align*}
		\Gamma_p\biggl(\biggl\{\frac{-nk}{p-1}\biggr\}\biggr)\omega^k(n^{-n})
		&=\omega^{p-1}\bigl(n^{1+\lfloor-nk/(p-1)\rfloor}\bigr)
		\,\frac{\prod_{\ell=0}^{n-1}\Gamma_p\bigl(\bigl\{-\frac k{p-1}+\frac\ell n\bigr\}\bigr)}{\prod_{\ell=1}^{n-1}\Gamma_p\bigl(\frac\ell n\bigr)}
		\\
		&=\frac{\prod_{\ell=0}^{n-1}\Gamma_p\bigl(\bigl\{-\frac k{p-1}+\frac\ell n\bigr\}\bigr)}{\prod_{\ell=1}^{n-1}\Gamma_p\bigl(\frac\ell n\bigr)},
		\end{align*}
		as $\omega^{p-1}$ is the trivial character. Now, combining this formula with the combinatorial identity~\eqref{eq:comb}
		and using
		$$
		\Gamma_p\biggl(\frac k{p-1}\biggr)^{\varphi(d)}
		=\frac{(-1)^{[k/(p-1)]_0\varphi(d)}}{\Gamma_p\bigl(1-\frac k{p-1}\bigr)^{\varphi(d)}}
		=(-1)^{[-k]_0\varphi(d)}\prod_{\substack{\ell=1\\(\ell,d)=1}}^d\frac1{\Gamma_p\bigl(1-\frac k{p-1}\bigr)},
		$$
		we transform the right-hand side in \eqref{Sd1} to the desired form.
	\end{proof}
	
	With Lemma \ref{lem:Sd} at our disposal, we can give a different expression for a special case of the finite hypergeometric function,
	an expression that resembles the classical hypergeometric function \eqref{eq:HGS}.
	
	\begin{Proposition}\label{prop:3}
		Assume that the hypergeometric data consist of $\l\in\Q^\times$, multi-sets $\balpha=\{\alpha_1,\dots,\alpha_m\}$ and $\bbeta=\{1,\dots,1\}$ closed under the Galois conjugation,
		and are primitive. Then for any finite field $\F_p$ with $p$ not dividing the least common denominator of $\alpha_1,\dots,\alpha_m$,
		\begin{equation}\label{eq:H4}
		H_p(\balpha, \bbeta;\l)=\frac1{\prod_{j=1}^m\Gamma_p(\alpha_j)}\,\frac{1}{1-p}
		\sum_{k=0}^{p-2}\frac{\prod_{j=1}^m\Gamma_p\bigl(\bigl\{\alpha_j-\frac k{p-1}\bigr\}\bigr)}{\Gamma_p\bigl(1-\frac{k}{p-1}\bigr)^m}
		\,(-1)^{km}(-p)^{\nu_{\balpha}(k)}\omega^k\bigl((-1)^m\l\bigr),
		\end{equation}
		where
		$$
		\nu_{\balpha}(k):=\sum_{j=1}^m\nu(k,\alpha_j(p-1)).
		$$
	\end{Proposition}
	
	\begin{proof}
		Represent $\balpha$ in the form \eqref{eq:partition} and apply Lemma \ref{lem:Sd} to get
		$$
		\prod_{i=1}^t S_{d_i}(\omega^k)
		=(-1)^{km}\prod_{j=1}^m\frac{\Gamma_p\bigl(\bigl\{\alpha_j-\frac k{p-1}\bigr\}\bigr)}
		{\Gamma_p(\alpha_j)\,\Gamma_p\bigl(1-\frac{k}{p-1}\bigr)}(-p)^{\nu(k,\alpha_j(p-1))}.
		$$
		The result then follows from Proposition~\ref{prop:2}.
	\end{proof}
	
	Remark that the factor $\prod_{j=1}^m\Gamma_p(\alpha_j)$ in \eqref{eq:H4} is a quadratic character and the formula looks particularly friendly when
	the hypergeometric data \eqref{hyp:main} are as in Theorem \ref{thm:main}:
	\begin{equation}\label{eq:Hp-pGamma}
	H_p(\balpha, \bbeta;1)=\frac{\chi_{\balpha}(p)}{1-p}
	\sum_{k=0}^{p-2}\frac{\prod_{j=1}^m\Gamma_p\bigl(\bigl\{\alpha_j-\frac k{p-1}\bigr\}\bigr)}{\Gamma_p\bigl(1-\frac{k}{p-1}\bigr)^m}
	\,(-p)^{\nu_{\balpha}(k)}.
	\end{equation}

	\section{Hypergeometric families of Calabi--Yau threefolds}\label{ss:CY}
	
	\subsection{Algebraic models of hypergeometric Calabi--Yau threefolds}\label{ss:models}
	
	Each $\balpha$ as stated in Theorem \ref{thm:main} corresponds to a one-parameter family of Calabi--Yau threefolds;
	see \cite[Table~1]{CYY} for their description and references. The construction was used in \cite{LT, BvS} for hypersurfaces (and complete intersections)
	in projective spaces and later was extended from projective spaces to weighted projective spaces \cite{Morrison, KT}.
	Here a weighted projective space is denoted by $\mathbb P^{N}(w_1,\dots,w_{N+1})$, where $N$ is the dimension and $w_i$ denotes the weight of the corresponding variable $X_i$,
	while $X(n_1,\dots,n_r)$ means the complete intersection of $r$ homogeneous polynomials of degree $n_1,\dots,n_r$, respectively.
	
	First, according to \cite{Morrison} there are four choices of $\balpha$ that correspond to degree $n=\sum_{i=1}^5 w_i$ homogeneous polynomials
	in $\mathbb P^4(w_1,\dots,w_5)$ of the form
	\begin{equation}
	V_\balpha(\psi):\quad w_1X_1^{e_1}+w_2X_2^{e_2}+w_3X_3^{e_3}+w_4X_4^{e_4}+ w_5X_5^{e_5}- n \psi X_1X_2X_3X_4 X_5=0.
	\label{Vaw}
	\end{equation}
	Compared to Table~1 of \cite{Morrison}, the weights $w_i$ are added to the terms ${X_i}^{e_i}$ so that the $\psi$-values of the desired fibers are normalized to be equal to~1.
	Each one is a one-parameter deformation of a Fermat type hypersurface $\sum_{i=1}^5 w_i{X_i}^{e_i}=0$ with a large automorphism group
	$$
	G_\balpha=\{(\zeta_{e_1}^{a_1}, \dots,\zeta_{e_5}^{a_5}):\zeta_{e_1}^{a_1} \cdots\zeta_{e_5}^{a_5}=1\}.
	$$
	For each $\psi\ne 0$, the fiber is a Calabi--Yau threefold.
	A classical paper on computing points of varieties like $V_{\balpha}(\psi)$ over finite fields is \cite{Koblitz} by Koblitz. { In \cite{Stienstra}, Stienstra associated a one-dimensional commutative formal group over $\Z$ to each $V_{\balpha}(\psi)$ with $\psi \in \Z$ in such a way that its formal logarithm 
		$$
		\sum_{n\geq 1} \frac{A_n(\psi,\balpha)}{n} \tau^n
		$$
		(like in equation (\ref{Antau}) of Section \ref{sec2.1}) satifsies, for almost all primes $p$, 
		$$
		A_p(\psi, \balpha)
			\equiv \PFQ{4}{3}{r_1,\,1-r_1,\,r_2,\,1-r_2}{\quad1,\quad1,\quad1}{\l}_{p-1} \mod p,
		$$
		where $\l$ is the parameter of the mirror family $\mathcal V_\balpha(\l)$.  It is convenient to call $A_p(\psi, \balpha)$ a generalized Hasse invariant of the variety $V_\balpha(\psi)$ at~$p$.
		A good prime $p$ is ordinary for $V_\balpha(\psi)$ when the value of the truncated hypergeometric series
		$$
		\PFQ{4}{3}{r_1,\,1-r_1,\,r_2,\,1-r_2}{\quad1,\quad1,\quad1}{\l}_{p-1}
		$$
		can be embedded into $\Z_p^\times$. 
	}
	\begin{table}[h]
		\caption{One-parameter families of hypersurfaces for $V_{\{r_1, r_2, 1-r_1,1-r_2\}}(\psi)$}
		\begin{center}
			\begin{tabular}{|c|c|c||c|c|}
				\hline
				$(d_1,\dots,d_t)$& $n$ &$(r_1,r_2)$&$X(n)\in \mathbb P^4(w_0,\dots,w_4)\vphantom{|^{0^0}}$&Calabi--Yau threefold equation\\
				\hline\hline
				$(5)$&5&$\bigl(\tfrac15,\tfrac25\bigr)$&$X(5)\subset \mathbb P^4(1,1,1,1,1)$&$\displaystyle \sum_ {j=1}^5 X_j^5-5\psi \prod_{j=1}^5 X_j=0$\\
				\hline
				$(10)$&10&$\bigl(\tfrac1{10},\tfrac3{10}\bigr)$&$X(10)\subset \mathbb P^4(1,1,1,2,5)$&$\displaystyle \sum_ {j=1}^3 X_j^{10}+2X_4^5+5X_5^2-10\psi \prod_{j=1}^5 X_j=0$\\
				\hline
				$(8)$&8&$\bigl(\tfrac18,\tfrac38\bigr)$&$X(8)\subset \mathbb P^4(1,1,1,1,4)$&$\displaystyle \sum_ {j=1}^4 X_j^8+4X_5^2-8\psi \prod_{j=1}^5 X_j=0$\\
				\hline
				$(3,6)$&6&$\bigl(\tfrac16,\tfrac13\bigr)$&$X(6)\subset \mathbb P^4(1,1,1,1,2)$&$\displaystyle \sum_ {j=1}^4 X_j^6+2X_5^3-6\psi \prod_{j=1}^5 X_j=0$\\
				\hline
			\end{tabular}
		\end{center}
		\label{tab3}
	\end{table}
	
	For other families, the defining equations are listed in Table~\ref{tab4}.
	They are computed in the following way. First, the forth column is listed in \cite{CYY} from which we know the number of equations,
	variables and the corresponding homogenous degree.
	Note that for the case $(\frac 1{12},\frac5{12})$ we use $X(12,12)$ in $\mathbb P^5(1,1,4,6,6,6)$ instead of $X(2,12)$ in $\mathbb P^5(1,1,1,1,4,6)$.
	Then we look for each equations of one-parameter deformation of a Fermat type such that under the de-homogenization\,---\,the map like
	\begin{equation}
	y_i=w_iX_i^{e_i} \;\;\text{for}\;\; i=1,\dots,5, \quad x_1=n\psi \prod_{i=1}^5 X_i, \quad \l=\psi^{-n},
	\label{qu-map}
	\end{equation}
	in the case of \eqref{Vaw}\,---\,they are sent to the mirrors whose equations can be worked out using \cite{BCM}; we outline this recipe below. Furthermore,
	we compute their generalized Hasse invariants as in \cite{Stienstra} and verify that they agree with the truncated hypergeometric functions when reduced modulo~$p$.
	
	\begin{table}[h]
		\caption{Complete intersection of one-parameter families of hypersurfaces for $V_{\{r_1, r_2, 1-r_1,1-r_2\}}(\psi)$}
		\label{tab4}
		\begin{center}
			{
				\begin{tabular}{|c|c||c|c|}
					\hline
					$(d_1,\dots,d_t)$&$(r_1,r_2)$&$X(n_1,\dots,n_r)\vphantom{|^{0^0}}$&Calabi--Yau threefold equations\\
					\hline\hline
					$(2,2,2,2)$&$\bigl(\tfrac12,\tfrac12\bigr)$&$X(2,2,2,2)\subset \mathbb P^7$&{\scriptsize$\begin{array}{c} X_1^2+X_2^2-2\psi X_3X_4=0\vphantom{|^{0^0}} \\X_3^2+X_4^2-2\psi X_5X_6=0  \\X_5^2+X_6^2-2\psi X_7X_8=0  \\X_7^2+X_8^2-2\psi X_1X_2=0  \end{array}$}\\
					\hline
					$(3,3)$&$\bigl(\tfrac13,\tfrac13\bigr)$&$X(3,3)\subset \mathbb P^5$&{\scriptsize$\begin{array}{c}X_1^3+X_2^3+X_3^3-3\psi X_4X_5X_6=0\vphantom{|^{0^0}} \\ X_4^3+X_5^3+X_6^3-3\psi X_1X_2X_3=0\end{array}$}\\
					\hline
					$(2,2,3)$&$\bigl(\tfrac12,\tfrac13\bigr)$&$X(2,2,3)\subset \mathbb P^6$&{\scriptsize$\begin{array}{c}X_1^2+X_2^2+X_3^2-3\psi X_4X_5=0\vphantom{|^{0^0}} \\ X_4^3+X_5^3-2\psi X_1X_6X_7=0 \\ X_6^2+X_7^2-2\psi X_2X_3=0\end{array}$}\\
					\hline
					$(2,2,4)$&$\bigl(\tfrac12,\tfrac14\bigr)$&$X(2,4)\subset \mathbb P^5$&{\scriptsize$\begin{array}{c}\displaystyle X_1^2+X_2^2+X_3^2+X_4^2-4\psi X_5X_6=0\vphantom{|^{0^0}} \\ \displaystyle X_5^4+X_6^4-2\psi X_1X_2X_3X_4=0\end{array}$}\\
					\hline
					$(12)$&$\bigl(\tfrac1{12},\tfrac5{12}\bigr)$&$X(12,12)\subset \mathbb P^5(1,1,4,6,6,6)$&{\scriptsize$\begin{array}{c}\displaystyle X_1^{12}+X_2^{12}-2\psi X_5X_6=0\vphantom{|^{0^0}} \\ \displaystyle X_5^2+X_6^2+4X_3^3+6X_4^2-12\psi X_1X_2X_3X_4=0\end{array}$}\\
					\hline
					$(4,4)$&$\bigl(\tfrac14,\tfrac14\bigr)$&$X(4,4)\subset \mathbb P^5(1,1,2,1,1,2)$&{\scriptsize$\begin{array}{c}\displaystyle X_1^4+X_2^4+2X_3^2-4\psi X_4X_5X_6=0\vphantom{|^{0^0}} \\ \displaystyle X_4^4+X_5^4+2X_6^2-4\psi X_1X_2X_3=0\end{array}$}\\
					\hline
					$(4,6)$&$\bigl(\tfrac14,\tfrac16\bigr)$&$X(4,6)\subset \mathbb P^5(1,1,2,1,2,3)$&{\scriptsize$\begin{array}{c}\displaystyle X_1^4+X_2^4+2X_3^2+2X_5^2-6\psi X_4X_6=0\vphantom{|^{0^0}} \\ \displaystyle X_4^6+3X_6^2-4\psi X_1X_2X_3X_5=0\end{array}$}\\
					\hline
					$(3,4)$&$\bigl(\tfrac13,\tfrac14\bigr)$&$X(3,4)\subset \mathbb P^5(1,1,1,1,1,2)$&{\scriptsize$\begin{array}{c}\displaystyle X_1^3+X_2^3+X_3^3+X_4^3-4\psi X_5X_6=0\vphantom{|^{0^0}} \\ \displaystyle X_5^4+2X_6^2-3\psi X_1X_2X_3X_4=0\end{array}$}\\
					\hline
					$(6,6)$&$\bigl(\tfrac16,\tfrac16\bigr)$&$X(6,6)\subset \mathbb P^5(1,2,3,1,2,3)$&{\scriptsize$\begin{array}{c}\displaystyle X_1^6+2X_2^3+3X_3^2-6\psi X_4X_5X_6=0\vphantom{|^{0^0}} \\ \displaystyle X_4^6+2X_5^3+3X_6^2-6\psi X_1X_2X_3=0\end{array}$}\\
					\hline
					$(2,2,6)$&$\bigl(\tfrac12,\tfrac16\bigr)$&$X(2,6)\subset \mathbb P^5(1,1,1,1,1,3)$&{\scriptsize$\begin{array}{c}\displaystyle  3 X_1^2+X_2^2+X_3^2+X_4^2-6\psi X_1X_5=0\vphantom{|^{0^0}} \\ \displaystyle X_5^6+X_6^2-2\psi X_2X_3X_4X_6=0\end{array}$}\\
					\hline
			\end{tabular}}
		\end{center}
	\end{table}
	
	There are different ways, such as the Griffiths--Dwork method, to compute the Picard--Fuchs equation of such a family $V_{\balpha}(\psi)$.
	A detailed computation was given in \cite{Morrison} by Morrison; the task can be alternatively done with the technology of GKZ hypergeometric functions~\cite{GKZ}. 
	For a discussion of computation using period integrals see the recent work \cite{HLYY} by Huang, Lian, Yau and Yu.
	
	One way to obtain the mirror of $V_\balpha(\psi)$ is via the orbifold $V_\balpha(\psi)/G_\balpha$ as described in Section~\ref{sec1}
	for the quintic case: one uses the map \eqref{qu-map}, which sends $V_\balpha(\psi)$ to
	$$
	\hat{\mathcal V}_{\balpha}(\l):\quad \sum_{i=1}^5 y_i=x_1, \quad M_n^{-1}\l x_1^{n}=\prod_{i=1}^5 y_i^{w_i}.
	$$
	Such maps can be given for other hypergeometric cases and a general recipe, related to the discussion in \cite[Section~6]{BCM}, is as follows.
	Using the partitioning \eqref{eq:partition} and Table \ref{tab2} write
	$$
	\frac{\prod_{j=1}^4(x-e^{2\pi ir_j})}{(x-1)^4}
	=\frac{(x^{p_1}-1)\cdots (x^{p_r}-1)}{(x^{q_1}-1)\cdots (x^{q_s}-1)}
	$$
	so that $\{p_i\}$ and $\{q_j\}$ are disjoint multi-sets of positive integers.
	Set $\l=\psi^{-(p_1+\dots+p_r)}$ and $(a_1,\dots,a_k)=(p_1,\dots,p_r,-q_1,\dots,-q_s)$, where $k=r+s$.
	For example, $(a_i)_{i=1,\dots,6}=(6,-1,-1,-1,-1,-2)$ when $\balpha=\{\frac 13, \frac 23,\frac 16,\frac 56 \}$.
	Next fragmentate the index set $\{1,2,\dots,k\}$ of $(a_1,a_2,\dots, a_k)$ into a disjoint union of subsets
	$K_1,\dots,K_m$ such that $\sum_{i\in K_j}a_i=0$ for each $j=1,\dots,m$ and each part $(a_i)_{i\in K_j}$ cannot be fragmented likewise further;
	we call such fragmentations of the index set minimal.
	In the example above, the index set $\{1,2,\dots,6\}$ is already a minimal fragmentation,
	while the index set $\{1,\dots,8\}$ of $(6,4,-1,-3,-1,-1,-2,-2)$ (corresponding to $\balpha=\{\frac 14, \frac 34,\frac 16,\frac 56 \}$)
	can be fragmented into $K_1=\{1,3,4,7\}$ and $K_2=\{2,5,6,8\}$ (this fragmentation is minimal but not unique\,---\,another minimal fragmentation
	can be found in the related row of Table~\ref{table: tab5}).
	Then there is a natural morphism from the corresponding homogeneous polynomial(s) in Tables~\ref{tab3} and \ref{tab4} to
	$$
	{W}_{\balpha}(\l): \quad \sum_{i\in K_1}x_i=\dots=\sum_{i\in K_m}x_i=0, \quad \l M^{-1}x_1^{a_1}\cdots x_k^{a_k}=(-1)^{q_1+\dots +q_j},
	$$
	where $M=M_{d_1}\dotsb M_{d_t}$. 
	Using the description of subsets $K_1,\cdots,K_m$ in Table \ref{table: tab5}, 
	the latter equation is the product of the $m$~equations 
	$$\psi^{-p_i}M_{p_i}^{-1} \prod_{j\in K_i}x_j^{a_j} =\prod_{g_j\in K_i}(-1)^{q_j}, \quad\text{where}\; i=1,\dots,m.$$ With $\l =\psi^{-(p_1+\cdots + p_r)}$, the mirror family is then given by the following set of $2m$ equations:
	$$
	\hat{\mathcal V}_{\balpha}(\l): \quad \sum_{j\in K_i}x_j=0, \quad \psi^{-p_i}M_{p_i}^{-1} \prod_{j\in K_i}x_j^{a_j} =\prod_{g_j\in K_i}(-1)^{q_j}, \quad\text{where}\; i=1,\dots,m.
	$$   We record the information for each minor family in Table~\ref{table: tab5};
	they are not unique\,---\,some instances differ from the descriptions in other sources, for example, in~\cite{Doran}. 
	
	In the literature, there are many papers on counting the number of solutions of these algebraic equations of type \eqref{Vaw} over finite fields; see, e.g.,
	\cite{DKSSVW, Kadir, Kadir-Yui, Perunicic, Wan}.
	In particular, it is shown in \cite{Wan} that $V_\balpha(\psi)$ and its mirror
	$\hat{\mathcal V}_\balpha(\l)$ share the same unit root.

	\begin{sidewaystable}
		\vskip160mm
		\caption{Mirror families $\hat{\mathcal V}_{\{r_1,r_2,1-r_1,1-r_2\}}(\l)$}
		\label{table: tab5}
		\begin{center}
			{
				\begin{tabular}{|c||c|c|c|c|}
					\hline
					$(r_1,r_2)$&$M$&$(a_1,\dots,a_k)$&fragmentation&map $V_{\balpha}(\psi)\to\mathcal V_{\balpha}(\l) \vphantom{|^{0^0}}$\\
					\hline\hline
					$\bigl(\tfrac15,\tfrac25\bigr)$& $5^5$ &{\scriptsize$( 5,-1,-1,-1,-1,-1)$}&{\scriptsize$\begin{array}{c}\{ 1,2,3,4,5,6\}\end{array}$}&
					{\scriptsize$\begin{array}{c}\vphantom{\Big|^{0^0}} x_1=-5\psi X_1\dotsb X_5,\; x_{i+1}=X_i^5\;\text{for}\; i=1,2,3,4,5,\;  \l=\psi^{-5}\end{array}$} \\
					\hline
					$\bigl(\tfrac1{10},\tfrac3{10}\bigr)$& $2^8 \cdot 5^5$ &{\scriptsize$( 10,-1,-1,-1,-2,-5)$}&{\scriptsize$\begin{array}{c}\{ 1,2,3,4,5,6 \}\end{array}$}&
					{\scriptsize$\begin{array}{c}\vphantom{|^{0^0}} x_1=-10\psi X_1\dotsb X_5,\; x_{i+1}=X_i^{10}\;\text{for}\; i=1,2,3,\\ x_5=2X_4^5,\; x_6=5X_5^2,\; \l=\psi^{-10}\end{array}$} \\
					\hline
					$\bigl(\tfrac18,\tfrac38\bigr)$& $2^{16}$ &{\scriptsize$( 8,-1,-1,-1,-1,-4)$}&{\scriptsize$\begin{array}{c}\{ 1,2,3,4,5,6 \}\end{array}$}&
					{\scriptsize$\begin{array}{c}\vphantom{\Big|^{0^0}} x_1=-8\psi X_1\dotsb X_5,\; x_{i+1}=X_i^{8}\;\text{for}\; i=1,2,3,4,\; x_6=4X_5^2,\; \l=\psi^{-8}\end{array}$} \\
					\hline
					$\bigl(\tfrac16,\tfrac13\bigr)$& $2^4\cdot 3^6$ &{\scriptsize$(6,-1,-1,-1,-1,-2)$}&{\scriptsize$\{1,2,3,4,5,6\}$}&
					{\scriptsize$\begin{array}{c}\vphantom{\Big|^{0^0}} x_1=-6\psi X_1\dotsb X_5,\; x_{i+1}=X_i^{6}\;\text{for}\; i=1,2,3,4,\; x_6=2X_5^3,\; \l=\psi^{-6}\end{array}$} \\
					\hline\hline
					$\bigl(\tfrac12,\tfrac12\bigr)$& $2^8$ &{\scriptsize$(2,2,2,2,-1,-1,-1,-1,-1,-1,-1,-1)$}&{\scriptsize$\begin{array}{c}\{1,5,6\}, \; \{2,7,8\}, \\ \{3,9,10\}, \; \{4,11,12\}\end{array}$}&
					{\scriptsize$\begin{array}{c}\vphantom{|^{0^0}} x_1=-2\psi X_3X_4,\; x_2=-2\psi X_5X_6,\; x_3=-2\psi X_7X_8,\\ x_4=-2\psi X_1X_2,\; x_{4+i}=X_i^2\;\text{for}\; i=1,\dots,8,\; \l=\psi^{-8}\end{array}$} \\
					\hline
					$\bigl(\tfrac13,\tfrac13\bigr)$& $3^6$ &{\scriptsize$(3,3,-1,-1,-1,-1,-1,-1)$}&{\scriptsize$\begin{array}{c}\{1,3,4,5\}, \\ \{2,6,7,8\}\end{array}$}&
					{\scriptsize$\begin{array}{c}\vphantom{|^{0^0}} x_1=-3\psi X_3X_4X_6,\; x_2=-3\psi X_1X_2X_3,\\ x_{2+i}=X_i^3\;\text{for}\; i=1,\dots,6,\; \l=\psi^{-6}\end{array}$} \\
					\hline
					$\bigl(\tfrac12,\tfrac13\bigr)$& $2^4\cdot 3^3$ &{\scriptsize$(3,2,2,-1,-1,-1,-1,-1,-1,-1)$}&{\scriptsize$\begin{array}{c}\{1,4,5,6\}, \\ \{2,7,8\}, \; \{3,9,10\}\end{array}$}&
					{\scriptsize$\begin{array}{c}\vphantom{|^{0^0}} x_1=-3\psi X_4X_5,\; x_2=-2\psi X_1X_6X_7,\; x_3=-2\psi X_2X_3,\\ x_{i+3}=X_{i}^3\;\text{for}\; i=1,2,3,4,5,\; x_9=X_6^2,\; x_{10}=X_7^2,\; \l=\psi^{-7}\end{array}$} \\
					\hline
					$\bigl(\tfrac12,\tfrac14\bigr)$& $2^{10}$ &{\scriptsize$(4,2,-1,-1,-1,-1,-1,-1)$}&{\scriptsize$\begin{array}{c}\{1,3,4,5,6\}, \\ \{2,7,8\}\end{array}$}&
					{\scriptsize$\begin{array}{c}\vphantom{|^{0^0}} x_1=-4\psi X_5X_6,\; x_2=-2\psi X_1X_2X_3X_4,\\ x_{i+2}=X_{i}^2\;\text{for}\; i=1,2,3,4,\; x_7=X_5^4,\; x_8=X_6^4,\; \l=\psi^{-6}\end{array}$} \\
					\hline
					$\bigl(\tfrac1{12},\tfrac5{12}\bigr)$& $2^{12}\cdot 3^6$ &{\scriptsize$(12,2,-1,-1,-4,-6,-1,-1)$}&{\scriptsize$\begin{array}{c}\{1,3,4,5,6\}, \\ \{2,7,8\}\end{array}$}&
					{\scriptsize$\begin{array}{c}\vphantom{|^{0^0}} x_1=-12\psi X_1X_2X_3X_4,\; x_2=-2\psi X_5X_6,\; x_3=X_1^{12},\; x_4=X_2^{12},\\ x_5=X_5^2,\; x_6=X_6^2,\; x_7=4X_3^3,\; x_8=6X_6^2,\; \l=\psi^{ -14}\end{array}$} \\
					\hline
					$\bigl(\tfrac14,\tfrac14\bigr)$& $2^{12}$ &{\scriptsize$(4,4,-1,-1,-2,-1,-1,-2)$}&{\scriptsize$\begin{array}{c}\{1,3,4,5\}, \\ \{2,6,7,8\}\end{array}$}&
					{\scriptsize$\begin{array}{c}\vphantom{|^{0^0}} x_1=-4\psi X_4X_5X_6,\; x_2=-4\psi X_1X_2X_3,\\ x_{i+2}=X_{i}^4\;\text{for}\; i=1,2,4,5,\; x_5=2X_3^2,\; x_{8}=2X_6^2,\; \l=\psi^{-8}\end{array}$} \\
					\hline
					$\bigl(\tfrac14,\tfrac16\bigr)$& $2^{10}\cdot 3^3$ &{\scriptsize$(6,4,-1,-3,-1,-1,-2,-2)$}&{\scriptsize$\begin{array}{c}\{1,5,6,7,8\}, \\ \{2,3,4\}\end{array}$}&
					{\scriptsize$\begin{array}{c}\vphantom{|^{0^0}} x_1=-6\psi X_4X_6,\; x_2=-4\psi X_1X_2X_3X_5,\; x_3=X_4^6,\; x_4=3X_6^2,\\ x_5=X_1^4,\; x_6=X_2^4,\; x_7=2X_3^2,\; x_8=2X_5^2,\; \l=\psi^{-10}\end{array}$} \\
					\hline
					$\bigl(\tfrac13,\tfrac14\bigr)$& $2^6\cdot 3^3$ &{\scriptsize$(4,3,-1,-1,-1,-1,-1,-2)$}&{\scriptsize$\begin{array}{c}\{1,3,4,5,6\}, \\ \{2,7,8\}\end{array}$}&
					{\scriptsize$\begin{array}{c}\vphantom{|^{0^0}} x_1=-4\psi X_5X_6,\; x_2=-3\psi X_1X_2X_3X_4,\\ x_{i+2}=X_i^3\;\text{for}\; i=1,2,3,4,\; x_7=x_5^4,\; x_8=2X_6^2,\; \l=\psi^{-7}\end{array}$} \\
					\hline
					$\bigl(\tfrac16,\tfrac16\bigr)$& $2^8\cdot 3^6$ &{\scriptsize$(6,6,-1,-2,-3,-1,-2,-3)$}&{\scriptsize$\begin{array}{c}\{1,3,4,5\}, \\ \{2,6,7,8\}\end{array}$}&
					{\scriptsize$\begin{array}{c}\vphantom{|^{0^0}} x_1=-6\psi X_4 X_5X_6,\; x_2=-6\psi X_1X_2X_3,\; x_3=X_1^6,\; x_4=2X_2^3,\\ x_5=3X_3^2,\; x_6=X_4^6,\; x_7=2X_5^3,\; x_8=3X_6^2,\; \l=\psi^{-12}\end{array}$} \\
					\hline
					$\bigl(\tfrac12,\tfrac16\bigr)$& $2^8\cdot 3^3$ &{\scriptsize$(6,2,-1,-1,-1,-1,-1,-3)$}&{\scriptsize$\begin{array}{c}\{1,5,6,7,8\}, \\ \{2,3,4\}\end{array}$}&
					{\scriptsize$\begin{array}{c}\vphantom{|^{0^0}} x_1=-6\psi X_1X_5,\; x_2=-2\psi X_2X_3X_4X_6,\\ x_{i+2}=X_i^6\;\text{for}\; i=1,2,3,4,\; x_{7}=X_5^6,\; x_{8}=3X_6^2,\; \l=\psi^{-8}\end{array}$} \\
					\hline
					
			\end{tabular}}
		\end{center}
		
	\end{sidewaystable}
	
	\subsection{Modularity}
	\label{ss:Modularity}
	For each of $\hat{\mathcal V}_{\balpha}(\l)$, the smooth model of the fiber
	at $\l=1$ corresponds to a Calabi--Yau threefold
	$\overline{\hat{\mathcal V}_{\balpha}(1)}$ defined over~$\Q$.
	
	We first state a theorem due to Katz which is relevant to our discussion.
	
	\begin{Theorem}[Katz]
		\label{th:Katz}
		Let $r_1,r_2$ be as above, $N$ be the least positive common denominators of $r_1,r_2$, and $\balpha=\{r_1,r_2,1-r_1,1-r_2\}$, $\bbeta=\{1,1,1,1\}$. Assume that $\ell$ is a fixed prime dividing $N$ and $\l \in \Z[1/N]$, $\l\ne0$. Take $K=\Q(\zeta_N)$. Then there is a continuous representation $\rho_{\balpha, \ell,\l}$ of $G_K$, of degree~$4$ when $\l\ne1$ and of degree~$3$ when $\l=1$, such that at each rational prime $p$ which splits completely in $\Z[\zeta_N]$, 
		\begin{equation}
		\operatorname{Tr} \rho_{\balpha, \ell,\l}(\Frob_p)=H_p(\balpha,\bbeta;\l).
		\end{equation}
		Moreover, at $\l=1$,  the representation $\rho_{\balpha, \ell,1}$ decomposes into a direct sum of $2$-dimensional and $1$-dimensional representations. 
	\end{Theorem}
	
	\begin{proof}
		Let $\balpha,\bbeta$ be as above, and let $k=\F_q$  be a finite field such that $q\equiv 1\mod N$,  $\psi$ a fixed nontrivial additive character of $k$. Let $k^{\text{sep}}$ denote a separable closure of~$k$. Katz constructed in \cite[Section~8.2]{Katz} an $\ell$-adic sheaf  $\mathcal F:=\mathcal H(\,!\,,\psi,\balpha,\bbeta)$  of $\mathbb G_m/k$, which is lisse on $\mathbb G_m-\{1\}$ of rank~4  and is self-dual, as both $\balpha,\bbeta$ are defined over $\Q$. \bk 
		
		At any geometric point $\bar \l $ of $\mathbb G_m(k^{\text{sep}})$ lying over $\l $, the stalk $\mathcal F_{\bar \l}$ of $\mathcal F$ gives rise to a generically degree~4 representation of $$\rho_{ \l,\ell,\F_q}\colon \text{Gal}(k^{\text{sep}}/k)\to GL(\mathcal F_{\bar \l}),$$ see \cite[Section  7.3.7]{Katz}. 	In particular, the  trace of the geometric Frobenius element  of $\text{Gal}(k^{\text{sep}}/k)$ acting on the representation space $\mathcal F_{\bar \l}$ agrees with $-H_q(\balpha,\bbeta;\l)^{K}$ which is recalled in \eqref{eq:Katz}, see  \cite[Section 8.2]{Katz}. These representations are compatible when one considers finite extensions $E$ of~$\F_q$.  They also  form a compatible family when $\l$ varies over $\mathbb A^1/\overline k$, where $\overline k$ is an algebraic closure of $k$. The function field of the algebraic variety $\mathbb A^1/\overline k$ is $\overline k(x)$ with an indeterminant $x$. Hence each stalk also admits the action of $\operatorname{Gal}\left (\overline k(x)^{\text{sep}}/\overline k(x)\right)$, which describes the coverings of $\mathbb A^1$ as a curve defined over $\overline k$.
		Each closed point $x$ of $\mathbb A^1$ can be viewed as a discrete valuation of $\overline k(x)/\overline k$. Fix a place $\tilde x$ of $\overline k(x)^{\text{sep}}$ lying over it and denote by $I(x)$ the inertia group of $\operatorname{Gal}\left (\overline k(x)^{\text{sep}}/\overline k(x)\right)$ at~$\tilde x$. 
		For each of the 14 cases, $\Lambda$ the product of the characters corresponding to upper and  lower parameters is trivial (notation as \eqref{eq:H}; this means that $\Lambda=\prod_{j=1}^4 \omega^{(q-1)\alpha_j}\prod_{j=1}^4 \omega^{-(q-1)\beta_j}$ is the trivial character).   From Theorem~8.4.2 in~\cite{Katz} and its proof we know that the inertia group $I(1)$  acts on $\mathcal F$  by a tame pseudoreflection (namely, its space of invariants has codimension~$1$, see \cite[Section 7.6]{Katz}). Consequently, the rank of $\mathcal F_{\bar 1}$, as a vector space over $\overline \Q_\ell$, has a $1$~dimension drop (see the proof of Theorem 8.4.11\,(3) in~\cite{Katz}), meaning that the rank becomes~$3$.  Moreover, from $\Lambda$ being trivial we know that the pseudoreflection at~1 is unipotent.  By the Parity Recognition Theorem 8.8.2 \cite[Theorem 8.8.2]{Katz}, the auto-duality pairing of $\mathcal F$ on the open subset of  $\mathbb G_m-\{1\}$ has to be alternating. Therefore, a nontrivial subspace of $\mathcal F_{\bar 1}$  also admits an alternating pairing, requiring the dimension of this subspace to be even, which can only be~2 in this case. \bk
		Thus, $\rho_{ 1,\ell,\F_q}$ has a decomposition into one $1$-dimensional and one $2$-dimensional representations of $\operatorname{Gal}(\F_q^{\text{sep}}/\F_q)$.  The subspace admitting alternating  pairing  has  pure weight $4+4-1$ (see Section~7.3.7 in~\cite{Katz} for the definition of weight and Theorem~8.4.2 for the claim).

		Later Katz modified $\mathcal H(\,!\,,\psi,\balpha,\bbeta)$ by a ``canonical'' twist\,---\,see \cite[Section~4]{Katz09}\,---\,to get  $\mathcal F^{\text{can}}:={\mathcal H}^{\text{can}}(\balpha;\bbeta)$. The trace of geometric Frobenius of $\F_q^{\text{sep}}/\F_q$ is given by $H_q(\balpha,\bbeta;\l)$. Due to the canonical twisting, the 2-dimensional subspace has pure weight~3 (see~\eqref{eq:weight} for the weight dropping by~4).
		The canonical twist removes the dependence on the additive character  $\psi$ of $\F_q$ and hence allows the canonical sheaf $\mathcal F^{\text{can}}$ to be lifted  from $\mathbb G_m$ over finite fields to over $\Z[\zeta_N]$; for the discussion on the inertia group $I(1)$, $\overline k$ can be lifted to an algebraic closure of $K=\Q(\zeta_N)$, while $I(1)$ lifted to the inertia group of $\operatorname{Gal}(\overline K(x)^{\text{sep}}/\overline K(x))$ at~1. 
		Consequently, the stalk of $\mathcal F^{\text{can}} $ at $\l$ gives rise to a generically 4-dimensional representation $\rho_{\balpha,\ell,\l}$ of~$G_K$. At $\l=1$, it is 3-dimensional  as the local monodromy at~$1$ is a unipotent pseudoreflection, which also forces the self-paring on a generic stalk to be alternating. Hence \bk  $\rho_{\balpha,\ell,1}$  decomposes as $\rho_{\balpha,\ell,1}^{(1)}\oplus \rho_{\balpha,\ell,1}^{(2)}$ of dimension~1 and~2, respectively, where the representation space of the latter  has pure weight~3.
	\end{proof}
	
	\begin{proof}[Proof of Theorem~\textup{\ref{thm:level}}]
		Let $\l=1$.
		For the hypergeometric data \eqref{hyp:main} we perform the partitioning \eqref{eq:partition}
		and introduce the corresponding cyclotomic extension $K=\Q(e^{2\pi i/d_1},\dots,e^{2\pi i/d_t})$.
		For each fixed prime $\ell$,  and $\l \in \Z[1/(\ell d_1 \cdots d_t)]$, we know from Theorem~\ref{th:Katz} that there is an $\ell$-adic Galois representation $\rho_{\balpha,\ell}:=\rho_{\balpha,\ell,1}$ of $G_K:=\operatorname{Gal}(\overline{\Q}/K)$ such that
		if $p\equiv 1\pmod{d_j}$ for $j=1,\dots, t$, then
		\begin{equation}\label{eq:HGM}
		H_p(\balpha,\bbeta;1)=\operatorname{Tr}\rho_{\balpha,\ell}(\Frob_p).
		\end{equation}
		Here $\Frob_p$ is the geometric Frobenius element at unramified prime $p$ (which splits completely in~$K$); in particular,
		\begin{equation}\label{eq:H2}
		H_q(\balpha,\bbeta;1)=\frac1{1-q}\sum_{\chi\in \widehat{\F_q^\times}}\prod_{j=1}^tS_{d_j}(\chi)
		\end{equation}
		as in \eqref{eq:H3} with $\omega^k=\chi$.

		Since the hypergeometric motives are defined over $\Q$, the Galois representation $\rho_{\balpha,\ell}$ can be extended to representations of the absolute
		Galois group $G_{\Q}$.%
		\footnote{Assume that $H$  is a finite index subgroup of a group $G$.  Up to semi-simplification, a representation $\rho$ of $H$  can be extended to a (non-unique) representation of $G$
			if and only if $\rho\sim \rho^g$ (its conjugate by $g$)  for all $g$ in $G/H$.}
		These extensions are not unique, but their irreducible components are only differed by finite order characters of $G_{\Q}$ which fix $G_K$.  \bk
		In \cite{BCM},
		Beukers, Cohen, and Mellit gave such an explicit extension, which is still denoted by $\rho_{\balpha,\ell}$, via Gauss sum
		properties and  a geometric realization using toric varieties compatible with Katz's formulation, which are $\hat{\mathcal V}_{\balpha}(1)$ in our cases. 
		The extended expression \eqref{eq:H2}  applies to almost all $\F_p$'s.

		From the established modularity lifting theorems  we know that the extension $\rho_{\balpha,\ell}^{(2)}$ to $G_\Q$ is modular.  For example, Theorem 2.1.4 of \cite{ALLL} states that given a prime $\ell$ and a $2$-dimensional absolutely irreducible representation of $G_\Q$ over $\overline \Q_\ell$ that is odd, unramified at almost all primes, and its restriction to the decomposition group $D_\ell$ at $\ell$ is crystalline with Hodge--Tate weight $\{0,r\}$, where $1\le r\le \ell-2$ and $\ell+1\nmid 2r$, then $\rho$ is modular and corresponds to a weight $r+1$ Hecke eigenform.  For our cases, we pick $\ell$ to be any prime larger than~$5$ and congruent to~$1$ modulo $N$. From Katz's construction, $\rho_{\balpha,\ell}^{(2)}$ is unramified almost everywhere, odd and absolutely irreducible due to the alternating pairing which is Galois invariant, and its restriction to $D_\ell$ is crystalline  with Hodge--Tate weight $\{0,3\}$.
		\bk This means there exists a weight~$4$ normalized Hecke eigenform $f_{\balpha}$ such that $\rho_{\balpha,\ell}^{(2)}$
		is isomorphic to the Deligne representation of $G_{\Q}$ associated with $f_{\balpha}$.  To determine the $f_{\balpha}$ individually, we compute the traces of $\rho_{\balpha,\ell}^{(2)}(\Frob_p)$, using the hypergeometric motives routines implemented 
		in \texttt{Magma} by Watkins via  $H_p(\balpha,\bbeta;1)$ defined by \eqref{eq:H2} (see \cite{Watkins} and \cite{RVRW}).
		As each $\rho_{\balpha,\ell}^{(2)}$ is unramifed outside the set $\{2,3,5,\ell\}$,   we use a result of Serre (Theorem~2.2 in \cite{Dieu-level}), which asserts that the $p$-exponents of the level
		are bounded by $8$ for $p=2$, by $5$ for $p=3$, and by $2$ for all other bad primes. \cy This theorem reduces our search to a finite list. We used \texttt{Magma} to compute the first few coefficients of all Hecke eigenforms of level either dividing $2^8\cdot 3^5$ or $2^8\cdot 5^2$, they are available at \href{https://sites.google.com/view/ft-tu/research/database?authuser=0}{\uline{the database}}.
		For the \texttt{Sagemath} program we used to identify the target modular forms, see the online \href{https://cocalc.com/share/public_paths/024b5733fc7476b0b78e27c3a82ef9351b4d03c0}{\uline{\text{Cocalc}}} file.
			This allowed us to confirm the levels of $f_{\balpha}$ listed in Table~\ref{tab1}.   Consequently, 	the extension of  $\rho_{\balpha,\ell}^{(1)}$ can be determined as well.  \cy For a closely related discussion, see a recent paper \cite{LLT} by Li, Long and Tu.\bk

			For example, when $(r_1,r_2)=(\frac 13,\frac 13)$, the level of $f_{\balpha}$ can only be 3, 9, 27, 81, or 243.
			Computing in \texttt{Magma} all weight~4 Hecke eigenforms with integer coefficients of such levels
			and comparing them with the explicit values of $H_p(\balpha,\bbeta;1)$ at $p=5$ and $7$ already identify $f_{\balpha}$
			as Entry 27.4.1.a in the database~\cite{LMFDB}.
			Other cases are verified in a similar fashion. 
		\end{proof}\bk
		
		We end this section with the following remark.
		As mentioned by Rodriguez-Villegas in \cite{RV}, the fiber at $\l=1$ for each of the fourteen hypergeometric families
		is a rigid Calabi--Yau threefold defined over $\Q$, and the modularity of rigid Calabi--Yau threefolds over $\Q$ is covered by the following theorem
		(see also \cite{Yui} for more background).
		
		\begin{Theorem}[Dieulefait \cite{Dieu},  Gouv\^ea and Yui \cite{GY}]\label{thm:mod}
			For each prime $\ell$, there is a weight $4$ modular form $f_{\balpha}$ with integer coefficient such that
			the $\ell$-adic Galois representation arising from the third \'etale cohomology
			group of $\hat{\mathcal V}_{\balpha}(1)$ is isomorphic
			to the $\ell$-adic Deligne representation associated to~$f_{\balpha}$.
		\end{Theorem}
		
		Therefore, it remains to provide an argument for verifying that
		$\overline{\hat{\mathcal V}_{\balpha}(1)}$ is indeed a rigid Calabi--Yau
		threefold over $\Q$.
		Our argument mainly follows the idea explained in \cite[Appendix]{WvG} (see also \cite{Meyer} for the quintic case via point counting).
		Let $h^{i,j}$ denote the Hodge numbers of $\overline{\hat{\mathcal V}_{\balpha}(1)}$.
		Denote by $N$ the least common denominator of $r_1$ and $r_2$.
		For a field $\F_q$ of characteristic $p>5$ such that $q\equiv 1\pmod{N}$,
		it follows from Theorem~6.1 in~\cite{BCM} that
		\begin{equation}\label{eq:count-pt}
		\#\overline{\hat{\mathcal V}_{\balpha}(1)}(\F_q)=-H_q(\balpha,\bbeta;1)+F_\balpha(q),
		\end{equation}
		where $F_{\balpha}(q)$ is a polynomial in $q$ with integer coefficients.
		By Weil's (ex-)conjectures, the local zeta function of $\overline{\hat{\mathcal V}_{\balpha}(1)}$ over $\F_q$ assumes the form
		$$
		\exp\biggl(\sum_{r=1}^{\infty}\#\overline{\hat{\mathcal V}_{\balpha}(1)}\frac{T^r}{r}\biggr)
		=\frac{f_3(T)}{(1-T)(1-qT)^{h^{1,1}}(1-q^2T)^{h^{2,1}}(1-q^3T)},
		$$
		where $f_3(T)$ is a polynomial of degree $2+2h^{2,1}$ with all
		roots of absolute value $q^{3/2}$. By \eqref{eq:count-pt}, there are only two roots of this absolute value,
		hence $h^{2,1}=0$ and we conclude that $\overline{\hat{\mathcal V}_{\balpha}(1)}$ is rigid.

		\section{Character-sum proof of Theorem~\ref{thm:main}}\label{ss:charsum}
		
		Throughout the section a prime $p>5$ is fixed.
		
		The finite hypergeometric function $H_p(\balpha,\bbeta;1)$, which is to be compared with
		the truncated $_4F_3$ hypergeometric sum, is defined by means of character sums.
		In this section we use its non-character-sum representation from Proposition~\ref{prop:3}
		and a methodology reminiscent to the one we had in Section~\ref{ss:Dwork}.
		A similar argument was used in \cite{Kilbourn, McCarthy-RV} to cover two cases of Theorem~\ref{thm:main}.
		
		As in Section~\ref{ss:Dwork}, given one of the fourteen multi-sets $\balpha=\{r_1,r_2,r_3,r_4\}$,
		we assume it labelled in such a way that the corresponding ``derivative'' multi-set $\{r_1',r_2',r_3',r_4'\}$
		is ordered: $r_1'\le r_2'\le r_3'\le r_4'$. This implies that the integers $a_j:=[-r_j]_0=pr_j'-r_j$ for $j=1,2,3,4$
		are ordered accordingly, $a_1\le a_2\le a_3\le a_4$,
		but also the pairing of the parameters:
		$$
		r_1+r_4=r_2+r_3=1, \quad r_1'+r_4'=r_2'+r_3'=1 \quad\text{and}\quad a_1+a_4=a_2+a_3=p-1.
		$$
		
		\begin{Lemma}
			\label{lem:red1}
			For $0\le k\le p-2$, the following congruences hold true modulo $p^3$\textup:
			\begin{equation*}
			\prod_{j=1}^4\frac{\G_p\bigl(\bigl\{r_j-\frac k{p-1}\bigr\}\bigr)(-p)^{\nu(k,r_j(p-1))}}{\G_p\bigl(r_j-\frac k{p-1}\bigr)}
			\equiv\begin{cases}
			1 & \text{if}\; 0\le k\le a_1, \\
			p\bigl(r_1'-\frac k{p-1}\bigr) & \text{if}\; a_1+1\le k\le a_2, \\
			p^2\bigl(r_1'-\frac k{p-1}\bigr)\bigl(r_2'-\frac k{p-1}\bigr) & \text{if}\; a_2+1\le k\le a_3, \\
			0 & \text{if}\; k\ge a_3+1.
			\end{cases}
			\end{equation*}
		\end{Lemma}
		
		\begin{proof}
			Observe that for each $j=1,2,3,4$, if $k\le r_j(p-1)$ then
			$$
			\G_p\biggl(\biggl\{r_j-\frac k{p-1}\biggr\}\biggr)(-p)^{\nu(k,r_j(p-1))}=\G_p\biggl(r_j-\frac k{p-1}\biggr);
			$$
			and if $r_j(p-1)<k<p-1$ then
			$$
			\G_p\biggl(\biggl\{r_j-\frac k{p-1}\biggr\}\biggr)(-p)^{\nu(k,r_j(p-1))}
			=-p\G_p\biggl(1+r_j-\frac k{p-1}\biggr)
			=p\biggl(r_j-\frac k{p-1}\biggr)^\star\G_p\biggl(r_j-\frac k{p-1}\biggr)
			$$
			where the factor in $(\,\cdot\,)^\star$ is omitted when divisible by $p$, that is, when $k=[-r_j]_0=a_j$.
			The latter can only happen when $a_j>r_j(p-1)$, equivalently, when $r_j<r_j'$ (hence $r_j=r_\ell'$ for some $\ell<j$ in view of the ordering).
			Thus, for a given $k$, $0\le k\le p-2$,
			\begin{align*}
			&
			\prod_{j=1}^4\frac{\G_p\bigl(\bigl\{r_j-\frac k{p-1}\bigr\}\bigr)(-p)^{\nu(k,r_j(p-1))}}{\G_p\bigl(r_j-\frac k{p-1}\bigr)}
			=\prod_{\substack{j=1\\r_j(p-1)<k}}^4p\biggl(r_j-\frac k{p-1}\biggr)^\star
			\displaybreak[2]\\ &\qquad
			=\prod_{\substack{j=1\\r_j'(p-1)<k}}^4p\biggl(r_j'-\frac k{p-1}\biggr)^\star
			=\prod_{\substack{j=1\\a_j+r_j-r_j'<k}}^4p\biggl(r_j'-\frac k{p-1}\biggr)^\star
			\displaybreak[2]\\ &\qquad
			=\begin{cases}
			1 & \text{if}\; 0\le k\le a_1, \\
			p\bigl(r_1'-\frac k{p-1}\bigr) & \text{if}\; a_1+1\le k<a_2, \\
			p\bigl(r_1'-\frac k{p-1}\bigr) & \text{if}\; k=a_2 \;\text{and}\; r_2\ge r_2', \\
			p^2\bigl(r_2'-\frac k{p-1}\bigr) & \text{if}\; k=a_2 \;\text{and}\; r_2<r_2', \\
			p^2\bigl(r_1'-\frac k{p-1}\bigr)\bigl(r_2'-\frac k{p-1}\bigr) & \text{if}\; a_2+1\le k<a_3, \\
			p^2\bigl(r_1'-\frac k{p-1}\bigr)\bigl(r_2'-\frac k{p-1}\bigr) & \text{if}\; k=a_3 \;\text{and}\; r_3\ge r_3', \\
			p^3\bigl(r_1'-\frac k{p-1}\bigr)\bigl(r_2'-\frac k{p-1}\bigr)\bigl(r_3'-\frac k{p-1}\bigr)/\bigl(r_3-\frac k{p-1}\bigr) & \text{if}\; k=a_3 \;\text{and}\; r_3<r_3', \\
			0\mod p^3 & \text{if}\; k\ge a_3+1.
			\end{cases}
			\end{align*}
			In the case $k=a_3$ and $r_3<r_3'$, we have either $r_1'-\frac k{p-1}$ or $r_2'-\frac k{p-1}$ divisible by $p$, so that
			$$
			p^2\biggl(r_1'-\frac k{p-1}\biggr)\biggl(r_2'-\frac k{p-1}\biggr)
			\equiv 0
			\equiv p^3\frac{\bigl(r_1'-\frac k{p-1}\bigr)\bigl(r_2'-\frac k{p-1}\bigr)\bigl(r_3'-\frac k{p-1}\bigr)}{r_3-\frac k{p-1}}
			\mod p^3;
			$$
			while in the case $k=a_2$ and $r_2<r_2'$, hence $r_2=r_1'$, we get
			\begin{align*}
			p^2\biggl(r_2'-\frac k{p-1}\biggr)
			&=p^2\biggl(r_2'-\frac{pr_2'-r_2}{p-1}\biggr)
			\\
			&=p\biggl(r_2-\frac{pr_2'-r_2}{p-1}\biggr)
			=p\biggl(r_1'-\frac k{p-1}\biggr)
			\end{align*}
			meaning that the expression obtained in this case agrees with the one for $k=a_2$ and $r_2\ge r_2'$.
			This completes the proof of the lemma.
		\end{proof}
		
		Using successively \eqref{eq:J-role} and \eqref{(t)_k} we have
		\begin{align*}
		&
		\frac{\prod_{j=1}^4\G_p\bigl(r_j-\frac k{p-1}\bigr)}{\G_p\bigl(1-\frac k{p-1}\bigr)^4\prod_{j=1}^4\G_p(r_j)}
		=\frac{\prod_{j=1}^4\G_p\bigl(r_j+k+\frac{kp}{1-p}\bigr)}{\G_p\bigl(1+k+\frac{kp}{1-p}\bigr)^4\prod_{j=1}^4\G_p(r_j)}
		\\ &\qquad
		\equiv\frac{\prod_{j=1}^4\G_p(r_j+k)}{\G_p(1+k)^4\prod_{j=1}^4\G_p(r_j)}
		\biggl(1+J_1(k)\,\frac{kp}{1-p}+J_2(k)\,\frac{(kp)^2}{(1-p)^2}\biggr)\mod p^3
		\\ &\qquad
		=\frac{\prod_{j=1}^4(r_j)_k}{k!^4\prod_{j=1}^4(pr_j')^{\nu(k,a_j)}}
		\biggl(1+J_1(k)\,\frac{kp}{1-p}+J_2(k)\,\frac{(kp)^2}{(1-p)^2}\biggr),
		\end{align*}
		where $J_1(k)$ and $J_2(k)$ are defined in~\eqref{J_12}.
		
		Combining the calculation and Lemma~\ref{lem:red1} we obtain the following: if $0\le k\le a_1$ then
		$$
		\frac{\prod_{j=1}^4\G_p\bigl(\bigl\{r_j-\frac k{p-1}\bigr\}\bigr)(-p)^{\nu(k,r_j(p-1))}}{\G_p\bigl(1-\frac k{p-1}\bigr)^4\prod_{j=1}^4\G_p(r_j)}
		\equiv\frac{\prod_{j=1}^4(r_j)_k}{k!^4}\biggl(1+J_1(k)\,\frac{kp}{1-p}+J_2(k)\,\frac{(kp)^2}{(1-p)^2}\biggr)\mod p^3;
		$$
		if $a_1+1\le k\le a_2$ then
		\begin{align*}
		\frac{\prod_{j=1}^4\G_p\bigl(\bigl\{r_j-\frac k{p-1}\bigr\}\bigr)(-p)^{\nu(k,r_j(p-1))}}{\G_p\bigl(1-\frac k{p-1}\bigr)^4\prod_{j=1}^4\G_p(r_j)}
		&\equiv\frac{\prod_{j=1}^4(r_j)_k}{k!^4}\,\frac{r_1'-\frac k{p-1}}{r_1'}\biggl(1+J_1(k)\,\frac{kp}{1-p}\biggr)\mod p^3
		\\
		&\equiv\frac{\prod_{j=1}^4(r_j)_k}{k!^4}\biggl(1+\frac k{r_1'}+\frac{kp}{r_1'}\biggr)(1+J_1(k)kp)\mod p^3;
		\end{align*}
		if $a_2+1\le k\le a_3$ then
		\begin{align*}
		\frac{\prod_{j=1}^4\G_p\bigl(\bigl\{r_j-\frac k{p-1}\bigr\}\bigr)(-p)^{\nu(k,r_j(p-1))}}{\G_p\bigl(1-\frac k{p-1}\bigr)^4\prod_{j=1}^4\G_p(r_j)}
		&\equiv\frac{\prod_{j=1}^4(r_j)_k}{k!^4}\,\frac{\bigl(r_1'-\frac k{p-1}\bigr)\bigl(r_2'-\frac k{p-1}\bigr)}{r_1'r_2'}\mod p^3
		\\
		&\equiv\frac{\prod_{j=1}^4(r_j)_k}{k!^4}\biggl(1+\frac k{r_1'}\biggr)\biggl(1+\frac k{r_2'}\biggr)\mod p^3;
		\end{align*}
		and if $k\ge a_3+1$ then
		$$
		\frac{\prod_{j=1}^4\G_p\bigl(\bigl\{r_j-\frac k{p-1}\bigr\}\bigr)(-p)^{\nu(k,r_j(p-1))}}{\G_p\bigl(1-\frac k{p-1}\bigr)^4\prod_{j=1}^4\G_p(r_j)}
		\equiv0
		\equiv\frac{\prod_{j=1}^4(r_j)_k}{k!^4}\mod p^3.
		$$
		
		Proposition \ref{prop:3} now implies
		\begin{equation*}
		H_p(\balpha,\bbeta;1)
		\equiv\sum_{k=0}^{p-1}\frac{\prod_{j=1}^4(r_j)_k}{k!^4}+\wt C_1\cdot\frac{1}{1-p}+\wt C_2\cdot\biggl(\frac{1}{1-p}\biggr)^2 \mod p^3,
		\end{equation*}
		where $\wt C_1$ and $\wt C_2$ are as in Lemma~\ref{cor:1}. From that lemma and also Lemma~\ref{lem:char} we therefore conclude that
		\begin{equation*}
		H_p(\balpha,\bbeta;1)
		\equiv\sum_{k=0}^{p-1}\frac{\prod_{j=1}^4(r_j)_k}{k!^4}+\chi_{\balpha}(p)\cdot p\mod p^3.
		\end{equation*}
		Thus, Theorem~\ref{thm:main} follows from comparison of this congruence with the equality~\eqref{eq:Haa}.
		Combining this result with Theorem~\ref{thm:level} one obtains for $p>5$,
		\begin{equation*} 
		F_1(\balpha)\equiv a_p(f_{\alpha})\mod p^3.
		\end{equation*}
		
		\section{Conclusion}
		\label{sec:7}
		
		For the convenience of readers, here we summarize our strategies
		used in this paper for establishing the supercongruences
		
		Our first  \cy result \bk is built on Dwork's work \cite{Dwork-padic}; especially, on Dwork's dash operation which allows us
		to derive the key reduction formula \eqref{eq:red}. The formula is then used to separate
		the first $p$-adic digit $a$ of a non-negative integer $k=a+bp$ in ratios $(r)_k/(1)_k$ of rising factorials
		and re-express the ratios by means of $(r)_a/(1)_a$ and explicit additional terms.
		Combined with Proposition \ref{prop:Dwork}, this leaves the task of showing that two particular coefficients,
		$C_1$ and $C_2$, are both congruent to $0$ modulo $p^3$. The latter is done by a residue sum calculation for
		certain rational functions and an execution of $p$-adic perturbation techniques.
		
		\cy In the character-sum approach to Theorem \ref{thm:main} \bk we use our Theorem \ref{thm:level} and first express
		the hypergeometric sums $H_p(\balpha,\bbeta;1)$ in the $p$-adic Gamma function form \eqref{eq:Hp-pGamma},
		in which the arguments are fractional parts, via the Gross--Koblitz formula.
		Then Lemma \ref{lem:red1} plays a role similar to that of formula \eqref{eq:Lambda} 
		and simplifies the expression; the newer version of $H_p(\balpha,\bbeta;1)$ can be further expanded in powers of $p/(1-p)$
		using local analyticity properties of the $p$-adic Gamma function. The remaining part is verifying that two coefficients
		in this expansion, $\tilde C_1$ and $\tilde C_2$, are both congruent to~0 modulo $p^3$, somewhat that has been already
		established earlier (in Lemma~\ref{cor:1}) as a companion to the congruences for $C_1$ and $C_2$.

		Among all fourteen weight 4 modular forms corresponding to the rigid hypergeometric Calabi--Yau threefolds,
		only $f_{\{\frac 14,\frac 34,\frac 13,\frac 23\}}$ is a CM modular form:
		$$
		a_p(f_{\{\frac 14,\frac 34,\frac 13,\frac 23\}})=J(\chi_3,\chi_3)^3+J(\chi_3^2,\chi_3^2)^3,
		$$
		when $p\equiv 1 \mod 3$ where $\chi_3$ stands for any cubic character of $\F_p$;
		so $p\equiv 1 \mod 3$ is ordinary and the unit root is $\gamma_p=-\G_p\bigl(\frac 13\bigr)^9$;
		for prime $p\equiv 2\mod 3$,  the Fourier coefficient simply vanishes.
		Numerically we observe that the supercongruence for this case also reflects the additional CM structure: data suggest  that for all primes $p\equiv 1\mod 3$,
		$$
		\PFQ{4}{3}{\frac 14,\,\frac 34,\,\frac 13,\,\frac 23}{1,\,1,\,1}{1}_{p-1}
		\equiv -\G_p\Bigl(\frac 13\Bigr)^9 \mod p^4,
		$$
		which is sharp.  In comparison,  the power $p^3$ is sharp in the statement of Theorem~\ref{thm:main}  for all other thirteen non-CM cases.
		Any proof of this observation may lead to new techniques in proving supercongruences of such kind.
		
		One may also hope to extend horizons of another powerful method of ``creative microscoping'' \cite{GZ19,GZ20} used for proving hypergeometric supercongruences related to quadratic characters rather than general Dwork's unit roots.
		The ideas behind the method are based on suitable ($q$-)deformations to replace $p$-adic perturbations, and missing ingredients at present are connections with combinatorics (``counting'') and geometry.
		Further development of this theme has potentials to also address the refined predictions from \cite{RRW} of Roberts and Rodriguez-Villegas mentioned in Section~\ref{sec2.1}.
		
		Finally, we would mention that the techniques developed in this paper are applicable to numerous other supercongruences of ``geometric'' or ``motivic'' origin;
		in particular, to the hypergeometric patterns observed in~\cite{RRW}.
		
		\section*{Acknowledgements}
		The authors would like to thank  the Banff International Research Station and the \text{MATRIX} Institute at the University of Melbourne in Creswick for excellent opportunities for collaboration.
		The authors are grateful to the enlightening discussions and interests of many colleagues including Frits Beukers, Charles Doran, Jerome Hoffman, Siu-Hung Ng, Robert Osburn, Ravi Ramakrishna, David Roberts, Fernando Rodriguez-Villegas, Duco van Straten, Masha Vlasenko, Mark Watkins and Jie Zhou. 
		The authors would like to further thank Fran\c cois Brunault
		and Michael Somos for supplying us with some explicit $\eta$-expressions in Table~\ref{tab1}, as well as
		the anonymous referees for several constructive suggestions.

		\appendix
		\section{The role of the unit root}
		\label{ss:A}
		
		\subsection{1-CFGL}
		\label{ss:A.1}
		Closely related to Dwork's congruence (Proposition \ref{prop:Dwork}) is the theory of 1-dimensional Commutative Formal Group Laws (1-CFGL) over a commutative ring $R$; see \cite[Appendix]{Stienstra-Beukers} for a review of this topic. A 1-CFGL is a formal power series $G(u,v)=u+v+\text{higher degree terms}\in R[[u,v]]$ that satisfies commutativity and associativity. It is determined by its logarithm which takes the form
		$$
		\ell(\tau)=\sum_{n\ge 1} \frac{b_n}{n}\tau^n,
		$$
		where $b_n\in R$ and $b_1=1$; namely, $G(u,v)=\ell^{-1}(\ell(u)+\ell(v))$. In \cite{Honda68}, Honda showed that for an elliptic curve $E$ defined over $\Z$ with  good reduction modulo $p$, the logarithm $\ell(\tau)$ of the associated formal group  can be taken as  the integral of  the normalized invariant differential of $E$ expanded at infinity.
		When $ b_p\not\equiv 0\pmod p$, the $p$-adic limit
		$$
		\tilde \gamma_p=\lim_{s\to \infty}\frac{b_{p^s}}{b_{p^{s-1}}}
		$$
		is the reciprocal of the unique invertable in~$\Z_p$ root of the denominator of the zeta function of~$E$.
		In view of the modularity of $E$, this $\gamma_p$ is a root of $T^2-a_p(f_E)T+p=0$ with $\gamma_p\not\equiv0\pmod p$, where $f_E$ is the weight~2 cuspidal Hecke eigenform associated with~$E$.
		Generalizations of Honda's result include \cite[Theorem 9]{Yui78} by the third author for formal Dirichlet series over an integral domain $R$ of characteristic~0. 	For $V_{\balpha}(\psi)$ listed in Table \ref{tab4}, the following theorem of Stienstra will be useful to us.
		
		\begin{Theorem}[Stienstra {\cite[Theorem 1]{Stienstra}}]\label{Stienstra}
			Let $K$ be a noetherian ring which is flat over $\Z$. Let $F_1,\dots,F_r$ be a regular sequence of homogenous polynomial in $K[T_0,\dots,T_N]$ and let $X$ be the scheme of $\P_K^N$ defined by the ideal $(F_1,\dots,F_r)$. Put $d_i=\deg F_i$ and $d=\sum_{i=1}^r d_i$. Assume that $X$ is flat over $K$ and $d_i\ge d-N\ge 1 $ for all~$i$. Let
			$$
			J=\{i=(i_0,\dots,i_N)\in \Z^{N+1} \mid i_0,\dots,i_N\ge 1, i_0+\dots i_N=d\}.$$
			Then there is a formal group law for $H^{N-r}(X,\hat G_{m,\mathcal O_X})$ over $K$ of dimension $n=\binom{d-1}N$ whose logarithm $\ell(\tau)$ is the $n$-tuple $(\ell_i(\tau))_{i\in J}$ of power series in $\tau=(\tau_i)_{i\in J}$ given by
			$$\ell_i(\tau)=\sum_{m\ge 1}\sum_{j\in J}m^{-1}\beta_{m,i,j}\tau_j^m,$$ where $\beta_{m,i,j}$ is the coefficient of $T_0^{mj_0-i_0}\cdots T_N^{mj_N-i_N}$ in $(F_1\cdots F_r)^{m-1}$.
		\end{Theorem}
		
		Here, given a formal group $G$, a sheaf $\mathcal J$ of $K$-algebras on $X$ and $i\in \Z_{\ge 0}$,  $H^i(X,G_\mathcal J)$ is a so-called Artin--Mazur  functor from nilpotent $K$-algebras to abelian groups; see \cite[Section~2]{Stienstra}. For the above, $G=\hat G_m$, the multiplicative formal group law  given by $\hat G_m(u,v)=u+v+uv$ for $u,v\in K$ and $\mathcal O_X$ is the structure sheaf  of  $X$.

		\subsection{1-CFGL and rigid Calabi--Yau threefolds}
		\label{ss:A.2}
		We now explain how to use Theorem~\ref{Stienstra} to show that $\gamma_p=\gamma_p(\balpha)$ in Proposition \ref{prop:Dwork} is related to the zeta function of a smooth model of $V(1)$.
		We will demonstrate this in the case of $(r_1,r_2)=(\frac12,\frac12)$. In this case,   $V(\psi)$ is arising from the intersection of four hypersurfaces. Other cases are verified similarly.  According to Table~\ref{tab4}, the threefold $V(\psi)$ is the intersection of four homogenous equations $f_1=f_2=f_3=f_4=0$, where
		\begin{alignat*}{2}
		f_1&=X_1^2+X_2^2-2 \psi X_3X_4, \quad& f_2&=X_3^2+X_4^2-2\psi X_5X_6,\\
		f_3&=X_5^2+X_6^2-2\psi X_7X_8, \quad& f_4&=X_7^2+X_8^2-2 \psi X_1X_2.
		\end{alignat*}
		Take $Y(\psi)=f_1f_2f_3f_4$. 
		
		The sequence $f_1,f_2,f_3, f_4$ is regular (see \cite[p.~184]{Hartshorne} for definition);
we claim that, for any fixed $\psi\in \Z$, the scheme is flat over $K=\Z[(2\psi)^{-1}]$, \cy which is a PID. \bk 
		
		Recall that over a PID, a module is flat if and only if it is torsion free.
If $\ff$ in $R=K[X_1,\dots, X_8]$ is a torsion of $R/I$, then  $r\cdot \ff\in I$ for some  $r\in K$; that is, $\ff$ can be written as  a linear combination of $f_i$ with coefficients in $\Q[X_1,\dots, X_8]$. Explicitly, we can show that these coefficients are in $R$, hence $\ff$ must be trivial in  $R/I$.  We first use $f_1$ or $(2\psi)^{-1}f_4$ to eliminate the presence of $X_1$ (and similarly of $X_2$ as they are interchangeable) from $\ff$ by replacing $X_1^2$ with $f_1-X_2^2+2\psi X_3X_4$.  After the elimination we write the reduced form of $\ff$ as $\ff_1+\ff_2$, where $\ff_1\in K[X_3,\dots, X_8] $ and $\ff_2\in K[X_1,\dots, X_8][f_1,\dots,f_4]$. Among them $\ff_1$ can be further written as a linear combination of $f_2,f_3$ with coefficients in $K[X_3,\dots,X_8]$. This allows us to record $\ff$  as a linear combination of $f_i$ with coefficients in $K[X_1,\dots,X_8]$. Thus, the scheme $X$ is indeed flat over the localized ring~$K$. 
				
		We next apply Theorem \ref{Stienstra} with $K=\Z[(2\psi)^{-1}]$. For the homogeneous polynomials $f_i$, degree  $d_i=\deg f_i=2$ and  $d=\sum_{i=1}^4d_i=8$. The dimension of the projective space is $N=8-1=7$, hence $d_i=2\ge 1=d-N$ for each $i$.
		The set $J$ is $\{(i_1,\dots, i_8)\in \Z^8: i_i\ge 1,\; i_1+\dots +i_8=d\}=\{(1,\dots,1)\}$. 
		By  Theorem~\ref{Stienstra}, $H^3(V(\psi), \hat G_{m, \mathcal O_{V(\psi)}})$ is a 1-CFGL, and the $m$-th coefficient $b_m(\psi)$ of its  logarithm is the coefficient of $(X_1\cdots X_8)^{m-1}$ of $(f_1\cdots f_4)^{m-1}$, namely,
		\begin{equation}\label{eq:b}
		b_m(\psi)={ 2^{4m-4}}\sum_{k= 0}^{m-1} \binom{m-1}{2k}^4 \frac{(\frac12)_k^4}{k!^4} \psi^{-8k}
		=2^{4m-4}\sum_{k\ge 0}^{m-1}\frac{(1-\frac{m}2)_k^4}{k!^4}\frac{(\frac12-\frac {m}2)_k^4}{k!^4}\psi^{-8k}.
		\end{equation}
		In particular, $b_p(1)\equiv F_1(\{\frac12,\frac12,\frac12,\frac12\}) \mod p$ when $p>2$ is a prime.
		Assume from now on that $b_p(1)\not\equiv 0\mod p$,  that is, $p$~is ordinary. Then by the theory of 1-CFGL \cite[Theorem A.8(v)]{Stienstra-Beukers},
		there exists a unique $\tilde \gamma_p\in \Z_p^\times$ such that
		\begin{equation}\label{eq:FGL}
		\tilde \gamma_p\equiv \frac{b_{p^s}(1)}{b_{p^{s-1}}(1)} \mod p^s
		\quad\text{for all}\; s\ge 1.
		\end{equation}
		
		Recall that by Proposition \ref{prop:Dwork}, for any \emph{ordinary} odd prime $p$ there is a $p$-adic unit root $\gamma_p=\lim_{s\to \infty} F_s/F_{s-1}$.  We use the following lemma to show that $\gamma_p=\tilde \gamma_p$.
		
		\begin{Lemma}\label{lem:2}
			Let $p$ be an odd prime such that $b_p(1)\not\equiv 0 \mod p$. Then for any fixed $m\ge 1$,
			$$
			\tilde \gamma_p=\lim_{s\to \infty}b_{p^s}(1)/b_{p^{s-1}}(1)\equiv \lim_{s\to \infty} F_s/F_{s-1}\equiv\gamma_p \mod p^m.
			$$
		\end{Lemma}
		
		\begin{proof}
			Let
			$$
			F(\l)=\sum_{k=0}^\infty \frac{(\frac12)_k^4}{k!^4}\l^k.
			$$
			By \cite[Theorems 2 and 3]{Dwork-padic}, for any integer $s\ge 1$,
			$$
			F(x)\cdot\PFQ{4}{3}{\frac 12 ,\, \frac 12 ,\, \frac 12 ,\, \frac 12}{1 ,\, 1 ,\, 1}{x^p}_{p^{s-1}-1}
			\equiv F(x^p)\cdot\PFQ{4}{3}{\frac 12 ,\, \frac 12 ,\, \frac 12 ,\, \frac 12}{1 ,\, 1 ,\, 1}{x}_{p^{s}-1}
			\mod p^s\Z_p[[x]],
			$$
			where the truncations of the hypergeometric sum are defined as in~\eqref{eq:4F3trunc}.
			These congruences define the quotient $F(\l)/F(\l^p)$ as a uniform $p$-adic analytic function $f(\l)$, which assumes unit values on
			$$
			D=\bigg\{z\in \C_p: |z|\le 1, \; \bigg|\PFQ{4}{3}{\frac 12 ,\, \frac 12 ,\, \frac 12 ,\, \frac 12}{1 ,\, 1 ,\, 1}{z}_{p-1}\bigg|=1\bigg\}, 
			$$
			where $|\,\cdot\,|$ is the norm on~$\C_p$. When $\l=1$ and $F_1\ne 0\mod p$, 
			$$
			f(1)=\gamma_p=\lim_{s\to\infty} \frac{F_s(1)}{F_{s-1}(1)}
			$$
			by Proposition~\ref{prop:Dwork}. 
			
			Let  $\l=\psi^{-8}$ and
			$$
			c_{p^s}(\l):=\sum_{k=0}^{\frac{p^s-1}2}\frac{(1-\frac{p^s}2)_k^4}{k!^4}\frac{(\frac12-\frac {p^s}2)_k^4}{k!^4}\l^k,
			$$
			so that $b_{p^s}(\psi)=2^{4(p^s-1)}c_{p^s}(\l)$ for $b_{p^s}(\psi)$ defined in~\eqref{eq:b}.
			Note that $2^{4(p^s-p^{s-1})}\equiv 1\mod p^s$. 
			Under the ordinary assumption, meaning that $\l\in D$, by the property of 1-CFGL (see \cite[Theorem A.8(v)]{Stienstra-Beukers} or \cite[Part~I, eq.~(12)]{BV}),  the $p$-adic limit
			$$
			\tilde \gamma(\l)=\lim_{s\to\infty } \frac{b_{p^s}(\psi)}{b_{p^{s-1}}(\psi^p)}
			$$
			exists and, for any $s\ge 1$,
			$$ \frac{b_{p^s}(\psi)}{b_{p^{s-1}}(\psi^p)}\equiv  \frac{c_{p^s}(\l)}{c_{p^{s-1}}(\l^p)} \mod p^s.$$
			Following  the analysis of \cite[Part I, Example 5.5]{BV}, we conclude that the congruences
			$$\frac{(\frac{1-p^s}2)_k}{k!}\equiv \frac{(\frac{1}2)_k}{k!} \mod p^{s-\text{ord}_p(k!)}, \quad   \frac{(1-\frac{p^s}2)_k}{k!}\equiv 1  \mod p^{s-\text{ord}_p(k!)},$$
			imply
			\begin{equation}\label{eq:15}
			c_{p^s}(\l)\equiv F(\l) \mod \big(\l^s,p^{\lfloor s\frac{p-2}{p-1}\rfloor}\big)\Z_p[[\l]]. 
			\end{equation}
			At the same time, we also have from 1-CFGL, for any fixed $m\ge 1$ and any $t,s\ge m$,
			$$\frac{c_{p^t}(\l)}{c_{p^{t-1}}(\l^p)}\equiv\frac{c_{p^s}(\l)}{c_{p^{s-1}}(\l^p)} \mod p^m.$$
			Since both $s$ (the degree of $\l$) and $\lfloor s\frac{p-2}{p-1}\rfloor$ (the exponent of~$p$) in the ideal $\big(\l^s,p^{\lfloor s\frac{p-2}{p-1}\rfloor}\big)\Z_p[[\l]]$ in \eqref{eq:15} are  increasing functions of $s$, by letting $s\to\infty$ in~\eqref{eq:15} we arrive at
			$$
			\frac{c_{p^t}(\l)}{c_{p^{t-1}}(\l^p)}
			\equiv \lim_{s\to \infty}\PFQ{4}{3}{\frac 12 ,\, \frac 12 ,\, \frac 12 ,\, \frac 12}{1 ,\, 1 ,\, 1}{\l}_{p^{s}-1}\bigg/\PFQ{4}{3}{\frac 12 ,\, \frac 12 ,\, \frac 12 ,\, \frac 12}{1 ,\, 1 ,\, 1}{\l^p}_{p^{s-1}-1} \mod p^m
			$$
			valid for any $m\ge1$.
			Thus,
			$$\lim_{t\to \infty}\frac{c_{p^t}(\l)}{c_{p^{t-1}}(\l^p)}=f(\l).$$ 
			Choosing $\l=1$, the lemma follows.
		\end{proof}

		\cy
		
		In what follows we continue to assume $\psi=1$.
		Note that the singularities of algebraic threefolds can be resolved using essentially blowups to get smooth threefolds over a field of characteristic~0 (see \cite{Zariski1944}) or $p>5$ (see \cite{Abhyankar66}). We denote by $X$ a smooth model of $V(1)$ over $\F_p$ of characteristic larger than~5. For the remaining of this section, 
		we outline how Stienstra's results \cite[\S\,3.6]{Stienstra-Formal} imply that $\gamma_p$ is related to the zeta function of~$X$ (hence, of~$V(1)$) via the crystalline cohomology groups $H^N_{\text{cris}}(X)$ (see \cite{Stienstra-Formal} for definition and notation).  A related discussion is available in the recent work \cite[Appendix]{BV} of Beukers and Vlasenko on the Dwork $F$-crystal.
		
		Recall from \cite[Chapter 3]{Stienstra-Formal} that the zeta function of $X$ over $\F_p$  is
		$$
		Z(X/\F_p;T)=\prod_{N=0}^{6}P_N(T)^{(-1)^{N+1}},
		$$
		where 
		$P_N(T)=\det\big(1-T F_p| H^N_{\text{cris}}(X)\otimes\Q\big)$ and $F_p$ is the Frobenius endomorphism on the de Rham--Witt complex on~$X$ (see \cite[\S\,3.4]{Stienstra-Formal}). 
		
		\begin{Proposition}
			\label{prop:St}
			In the notation above, $P_3(1/\gamma_p)=0$.
		\end{Proposition}
		
		\begin{proof}
			We give the steps of how the ingredients of Stienstra's work \cite{Stienstra, Stienstra-Formal} combine together. 
			
			\smallskip\noindent
			(i)  Stienstra relates the CFGL groups of $Y(1)=0$ (defined by the hypersurface $f_1\cdots f_4=0$) and of $V(1)$ (the complete intersection of $f_1=0,\dots,f_4=0$). Let $f\colon V(1)\to\P^7$ be the inclusion map, $\mathcal O_{V(1)}$ structure sheaf of $V(1)$ and $\mathcal F=f_* \mathcal O_{V(1)}$. For any non-empty proper subset $\rho\subset\{1,2,3,4\}$, denote $Y_\rho=\prod_{i\in \rho} f_i$; then $Y_\rho=0$ is a variety of $\P^7$ containing $V(1)$. Let $\mathcal F_\rho$ be the corresponding sheaf defined as for~$\mathcal F$.  In \cite[Section~4]{Stienstra}, Stienstra shows the following two claims (see  \cite[Lemma 4.5  and (4.6.1)]{Stienstra}):
			\begin{gather}
			\label{eq:6}
			H^i(\mathbb P^7, \hat G_{m,\mathcal F_\rho})=0 \quad \text{for}\; i=1,\dots,7,
			\\
			\label{eq:7}
			H^7(\mathbb P^7, \hat G_{m, \mathcal F}) \cong H^{3}(V(1),\hat G_{m,\mathcal O_{V(1)}}).
			\end{gather}
			As the unit root $\gamma_p$ is computed from the formal logarithm of $H^7(\mathbb P^7, \hat G_{m, \mathcal F})$, it can be computed from the formal logarithm of $H^{3}(V(1),\hat G_{m,\mathcal O_{V(1)}})$. 
			
			\smallskip\noindent
			(ii) To pass $\gamma_p$ from $H^{3}(V(1),\hat G_{m,\mathcal O_{V(1)}})$ to $H^{3}(X,\hat G_{m,\mathcal O_{X}})$,
					it is sufficient \cite[\S\,2.5]{Stienstra-Formal} to show that the singularities of $V(1)$ are \emph{rational} singularities (see \cite[Definition 1]{Viehweg}).
			(In our situation, the singularities of $V(1)$ are determined by the Jacobian matrix 
			$$
			\begin{pmatrix} 
			2X_1&2X_2&-2X_4&-2X_3&0&0&0&0\\
			0&0& 2X_3&2X_4&-2X_6&-2X_5&0&0\\
			0&0&0&0& 2X_5&2X_6&-2X_8&-2X_7\\
			-2X_2&-2X_1&0&0&0&0& 2X_7&2X_8
			\end{pmatrix}
			$$
			of rank less than four.
			The Jacobian gives rise to 96 isolated solutions in $\C\P^7$ including, for example, $[1,1,\dots,1]$.  
			The set of singularities  can be verified using \texttt{Magma} package \texttt{PrimaryComponents(SingularSubscheme(V))}.)
			The singularities of $V(1)$ only consist of isolated ordinary double points, hence are rational singularities. They can be resolved individually by blowing up (see, for example, \cite[Section~1.4]{Hartshorne} and \cite{Zariski1944}).  
			
			\smallskip\noindent
			(iii) Finally, in \cite[Section 3]{Stienstra-Formal}, Stienstra explains how to go from  $H^{3}(X,\hat G_{m,\mathcal O_{X}})$ to  $H_{\text{cris}}^3(X)$.
			The particular relevant to us result is as follows.
			
			\begin{Theorem}[Stienstra]
				Let $X$ be a smooth projective variety over $\F_p$ such that  $H^N(X,\hat G_{m,X})$ is a \textup{1-CFGL} over $\Z_p$ with formal logarithm
				$$\sum_{m\ge 1}m^{-1} \beta_m \tau^m.
				$$
				Take
				$$
				P_N(T)=\det\big(1-T F_p| H^N_{\operatorname{cris}}(X)\otimes\Q\big), 
				$$
				where $F_p$ is  the Frobenius endomorphism on the de Rham--Witt complex on $X$. Then 
				$$
				P_N(T)=1+a_1T+pa_2T+\dots +p^{k-1}a_kT^k \quad\text{with}\; a_1,\dots,a_k\in \Z,
				$$ where  $k=\deg P_N(T)$ is the $N$-th Betti number of~$X$, and for all integers $m,n\ge 1$,
				\begin{equation}\label{eq:Stienstra-CGL}
				\beta_{mp^n}+a_1\beta_{mp^{n-1}}+ pa_2\beta_{mp^{n-2}}+\dots+ p^{k-1}a_k\beta_{mp^{n-k}}\equiv 0\mod p^n .
				\end{equation} 
			\end{Theorem}
			
			\smallskip\noindent
			(iv) When $N=3$, we know that $H^{3}(X,\hat G_{m,\mathcal O_{X}})$ is a 1-CFGL.
			If $\beta_p\not\equiv 0\mod p$, i.e.\ $p$ is ordinary, by \cite[Theorem A.8(v)]{Stienstra-Beukers},
			for any $m,n\ge 1$,
			\begin{equation}
			\frac{\beta_{mp^n}}{\beta_{mp^{n-1}}}\equiv\gamma_p \mod p^n ,
			\label{2 term cong}
			\end{equation}
			where $\gamma_p\in \Z_p^\times$. Comparing with \eqref{eq:Stienstra-CGL}, when $m=n=1$, this implies that $\gamma_p=-a_1 \mod p$. When $m=1$, $n=2$, the congruence~\eqref{eq:Stienstra-CGL} translates after division by $\beta_p$ into
			$$
			\frac{\beta_{p^2}}{\beta_p}+a_1+pa_2\frac{\beta_1}{\beta_p}
			\equiv\gamma_p+a_1+\frac{pa_2}{\gamma_p}
			\equiv 0 \mod p^2,
			$$
			where \eqref{2 term cong} was used, resulting in
			$$\gamma_p^2+a_1\gamma_p+pa_2\equiv 0 \mod p^2.$$
			Continuing this inductively we deduce, for any integer $n\ge 0$,
			$$\gamma_p^k+a_1\gamma_p^{k-1}+pa_2\gamma_p^{k-2}+\dots +p^{k-1}a_k\equiv 0 \mod p^{k+n}.$$
			Thus, $P_3(1/\gamma_p)=0$ as required.     
		\end{proof}      
		
		The other thirteen cases are processed similarly.  
		
		\subsection{The formal logarithms}
		\label{ss:A.3}
		In Table~\ref{tab6} we list the coefficients of the formal logarithms
		$$
		\sum_{m\ge 1}m^{-1}b_m(\psi,\balpha) \tau^m
		$$
		for all fourteen cases $V_{\balpha}(\psi)$. \bk
		
		\begin{table}[h] 
			\caption{Coefficients $b_m(\psi,\balpha)$ of the formal logarithms}
			\begin{center}
				\begin{tabular}{|c||c|c|}
					\hline
					$(r_1,r_2)$&$b_m(\psi,\balpha)$&$\lambda$\\
					\hline\hline
					$\bigl(\tfrac15,\tfrac25\bigr)$& $\displaystyle(-5\psi)^{m-1} \sum_{k\ge 0}\binom{m-1}{5k}\frac{\vphantom{\big|^0} (\frac15)_k(\frac{2}5)_k(\frac{3}5)_k(\frac 45)_k}{k!^4}{(-1)^{5k}}\l^k$&$\psi^{-5}$\\
					\hline
					$\bigl(\tfrac1{10},\tfrac3{10}\bigr)$& $\displaystyle(-10\psi)^{m-1} \sum_{k\ge 0}\binom{m-1}{10k}\frac{\vphantom{\big|^0} (\frac1{10})_k(\frac3{10})_k(\frac{7}{10})_k(\frac 9{10})_k}{k!^4}\l^k$&$\psi^{-10}$\\
					\hline
					$\bigl(\tfrac18,\tfrac38\bigr)$& $\displaystyle(-8\psi)^{m-1} \sum_{k\ge 0}\binom{m-1}{8k}\frac{\vphantom{\big|^0} (\frac1{8})_k(\frac3{8})_k(\frac{5}{8})_k(\frac 7{8})_k}{k!^4}\l^k$&$\psi^{-8}$\\
					\hline
					$\bigl(\tfrac16,\tfrac13\bigr)$& $\displaystyle(-6\psi)^{m-1} \sum_{k\ge 0}\binom{m-1}{6k}\frac{\vphantom{\big|^0} (\frac1{6})_k(\frac1{3})_k(\frac{2}{3})_k(\frac 5{6})_k}{k!^4}\l^k$&$\psi^{-6}$\\ \hline
					$\bigl(\tfrac12,\tfrac12\bigr)$& $\displaystyle(2\psi)^{4m-4} \sum_{k\ge 0}\binom{m-1}{2k}^4\,\frac{\vphantom{\big|^0} (\frac1{2})_k^4}{k!^4}\l^k$&$\psi^{-8}$\\
					\hline
					$\bigl(\tfrac13,\tfrac13\bigr)$& $\displaystyle(3\psi)^{2m-2} \sum_{k\ge 0}\binom{m-1}{3k}^2\,\frac{\vphantom{\big|^0} (\frac1{3})_k^2(\frac 23)_k^2}{k!^4}\l^k$&$\psi^{-6}$\\
					\hline
					$\bigl(\tfrac12,\tfrac13\bigr)$& $\displaystyle(-12\psi^3)^{m-1} \sum_{k\ge 0}\binom{m-1}{2k}^2\binom{m-1}{3k}\frac{\vphantom{\big|^0} (\frac1{3})_k(\frac12)_k^2(\frac 23)_k}{k!^4}(-\l)^k$&$\psi^{-7}$\\
					\hline
					$\bigl(\tfrac12,\tfrac14\bigr)$& $\displaystyle{(8\psi^2)^{m-1}} \sum_{k\ge 0}\binom{m-1}{2k}\binom{m-1}{4k}\frac{\vphantom{\big|^0} (\frac1{4})_k(\frac1{2})_k^2(\frac 3{4})_k}{k!^4}\l^k$&$\psi^{-6}$\\
					\hline
					$\bigl(\tfrac1{12},\tfrac5{12}\bigr)$& $\displaystyle(24\psi^2)^{m-1} \sum_{k\ge 0}\binom{m-1}{2k}\binom{m-1}{12k}\frac{\vphantom{\big|^0} (\frac1{12})_k(\frac5{12})_k(\frac 7{12})_k(\frac{11}{12})_k}{k!^4}\l^k$&$\psi^{-14}$\\
					\hline
					$\bigl(\tfrac1{4},\tfrac1{4}\bigr)$& $\displaystyle(4\psi)^{2m-2} \sum_{k\ge 0}\binom{m-1}{4k}^2\,\frac{\vphantom{\big|^0} (\frac1{4})_k^2(\frac 3{4})_k^2}{k!^4}\l^k$&$\psi^{-8}$\\
					\hline
					$\bigl(\tfrac1{4},\tfrac1{6}\bigr)$& $\displaystyle(24\psi^2)^{m-1} \sum_{k\ge 0}\binom{m-1}{4k}\binom{m-1}{6k}\frac{\vphantom{\big|^0} (\frac1{6})_k(\frac1{4})_k(\frac 3{4})_k(\frac{5}{6})_k}{k!^4}\l^k$&$\psi^{-10}$\\
					\hline
					$\bigl(\tfrac1{3},\tfrac1{4}\bigr)$& $\displaystyle(12\psi^2)^{m-1} \sum_{k\ge 0}\binom{m-1}{3k}\binom{m-1}{4k}\frac{\vphantom{\big|^0} (\frac1{4})_k(\frac1{3})_k(\frac 2{3})_k(\frac{3}{4})_k}{k!^4}(-\l)^k$&$\psi^{-7}$\\
					\hline
					$\bigl(\tfrac1{6},\tfrac1{6}\bigr)$& $\displaystyle(6\psi)^{2m-2} \sum_{k\ge 0}{\binom{m-1}{6k}}^2\,\frac{\vphantom{\big|^0} (\frac1{6})_k^2(\frac 5{6})_k^2}{k!^4}\l^k$&$\psi^{-12}$\\ \hline
					$\bigl(\tfrac1{2},\tfrac1{6}\bigr)$& $\displaystyle(12\psi^2)^{m-1} \sum_{k\ge 0}\binom{m-1}{2k}\binom{m-1}{6k}\frac{\vphantom{\big|^0} (\frac1{6})_k(\frac12)_k^2(\frac 56)_k}{k!^4}\l^k$ 
					&$\psi^{-8}$\\ \hline
				\end{tabular}
			\end{center}
			\label{tab6}
		\end{table}

	\end{document}